\documentclass[12pt]{article}
\usepackage{amsmath, amssymb, amsthm, bm}
\usepackage{graphicx}
\usepackage{enumerate}
\usepackage{mathtools}
\usepackage{dsfont}
\usepackage{verbatim}
\usepackage{graphicx,rotating}
\usepackage{xcolor}
\usepackage{enumitem} 
\usepackage{natbib}
\usepackage{float}

\usepackage{xr}
\usepackage{url} 

\newcommand{\blind}{1}

\newcommand{\R}{\mathbb{R}}

\renewcommand{\H}{\mathcal{H}}
\renewcommand{\L}{\mathcal{L}}
\newcommand{\bSigma}{\mathbf{\Sigma}}
\newcommand{\I}{\mathbf{I}}
\newcommand{\e}{\mathbf{e}}
\newcommand{\X}{\mathbf{X}}
\newcommand{\Y}{\mathbf{Y}}
\newcommand{\bbeta}{\bm{\beta}}
\newcommand{\bepsilon}{\bm{\epsilon}}

\renewcommand{\P}{{\rm P}}
\newcommand{\E}{{\rm E}}
\newcommand{\Var}{{\rm Var}}
\newcommand{\Cov}{{\rm Cov}}
\newcommand{\EEC}{{\rm EEC}}
\newcommand{\FWER}{{\rm FWER}}
\newcommand{\CER}{{\rm CER}}

\newtheorem{thm}{Theorem}
\newtheorem{lemma}{Lemma}
\newtheorem{corr}{Corollary}

\theoremstyle{definition}
\newtheorem{defn}{Definition}

\theoremstyle{remark}

\newtheorem*{remark}{Remark}

\AtBeginDocument{%
 \abovedisplayskip=6pt plus 3pt minus 9pt
 \abovedisplayshortskip=0pt plus 3pt
 \belowdisplayskip=6pt plus 3pt minus 9pt
 \belowdisplayshortskip=3pt plus 3pt minus 4pt
}

\begin{document}

\def\spacingset#1{\renewcommand{\baselinestretch}%
{#1}\small\normalsize} \spacingset{1}


\if1\blind
{
  \title{\bf Estimation of Expected Euler Characteristic Curves of Nonstationary Smooth Gaussian Fields}
  \author{Fabian J.E. Telschow\thanks{
    The authors F.T., D.C. and A.S. were partially supported by NIH grant \textit{R01EB026859}}\hspace{.2cm}\\
    Division of Biostatistics, University of California, San Diego\\
    Armin Schwartzman \\
    Division of Biostatistics and Hal{\i}cio\u{g}lu Data Science Institute,\\ University of California, San Diego\\
    Dan Chang \\
    School of Mathematical and Statistical Sciences, Arizona State University\\ University of California, San Diego \\
    and \\
    Pratyush Pranav \\
    Univ. Lyon, ENS de Lyon, Univ. Lyon 1, CNRS and Centre de Recherche\\ Astrophysique de Lyon (UMR5574)
    }
  \maketitle
} \fi

\if0\blind
{
  \bigskip
  \bigskip
  \bigskip
  \begin{center}
    {\LARGE\bf Estimation of Expected Euler Characteristic Curves of Nonstationary Smooth Gaussian Fields}
\end{center}
  \medskip
} \fi

\bigskip
\begin{abstract}
The expected Euler characteristic (EEC) of excursion sets of a Gaussian random field over a compact manifold approximates the distribution of its supremum for high thresholds. Viewed as a function of the excursion threshold, the EEC of a Gaussian mean zero, constant variance field is expressed by the Gaussian kinematic formula (GKF) as a linear function of the Lipschitz-Killing curvatures (LKCs) of the field, which solely depend on the covariance function and the manifold.
In this paper, we propose consistent estimators of the LKCs as linear projections of ``pinned" Euler characteristic (EC) curves obtained from realizations of a functional Gaussian multiplier bootstrap. We show that these estimators have low variance, can estimate the LKCs of the limiting field of non-stationary, non-Gaussian fields satisfying a functional CLT, handle unknown means and variances of the observed fields and can be computationally efficiently implemented even for complex underlying manifolds.
Furthermore, a parametric plugin-estimator for the EEC curve of the limiting field is presented, which is more efficient than the nonparametric average of EC curves. The proposed methods are evaluated using simulations of 2D fields, and illustrated on cosmological observations and simulations on the 2-sphere.
\end{abstract}

\noindent%
{\it Keywords:}  Random field theory, Lipschitz-Killing curvatures, Gaussian related fields
\vfill

\newpage
\spacingset{1.5} 
\section{Introduction}
\label{sec:Introduction}

\paragraph{The expected Euler characteristic curve}
The expected Euler characteristic curve (EEC) of a random field is a function that describes the EEC of the excursion sets of the field as a function of the excursion threshold. For large thresholds, the EEC curve is an excellent approximation of the tail distribution of the supremum of a mean zero, variance one, smooth Gaussian fields defined over a compact domain \citep{Taylor:2005}. For this reason, the EEC curve has been extensively used to set the significance threshold for control of the family-wise error rate (FWER), particularly in neuroimaging studies \citep{Worsley:1996,Worsley:2004,Nichols:2012}. For the same reason, it has recently been used to construct simultaneous confidence bands for functional data in  \citep{Telschow:2019,Telschow:2019b, Liebl:2019}.

The strength of the EEC curve lays in the fact that it can be written explicitly for Gaussian (and certain Gaussian-related fields) by means of the Gaussian kinematic formula (GKF) \citep{RFG:2007, Taylor:2006}. It ingeniously connects geometric and probabilistic properties of a smooth mean zero, variance one, Gaussian field $f$, which is defined over a compact $d$-dimensional manifold $S$, possibly with piecewise $\mathcal{C}^2$-boundary. It states that the EEC of the excursion set $A_f(u) = \{s \in S: f(s) \ge u\}$ can be written as
\begin{equation}
	\label{eq:EEC}
	\EEC(u) = \E\big[\chi\big(A_f(u)\big)\big] = \L_0 \Phi^+(u) + \sum_{d=1}^D \L_{d} \rho_d(u)\,,
\end{equation}
which remarkably is a finite linear combination of the so-called EC-densities
\begin{equation}
\label{eq:EC-densities}
\rho_d(u) = (2\pi)^{-(d+1)/2} H_{d-1}(u) e^{-u^2/2}, \qquad d=1,\ldots,D,
\end{equation}
where $H_d$ is the $d$-th probabilistic Hermite polynomial and $\Phi^+(u) = \P(N(0,1) > u)$. The linear coefficients $\L_0,\ldots,\L_D$ are called the \textit{Lipschitz-Killing curvatures} (LKCs) of $S$ and are intrinsic volumes of $S$ considered as a Riemannian manifold endowed with a Riemannian metric induced by $f$, c.f. \citep[Chapter 12]{RFG:2007}. In applications the EC densities are usually known and only the LKCs, except for $\L_0$, which is simply the EC of $S$, need to be estimated, since they depend on the unknown correlation function of $f$.

\paragraph{Previous work on estimation of LKCs}
Estimation of LKCs was first studied in the neuroimaging community assuming that the random field is stationary isotropic \citep{Worsley:1992, Worsley:1996, Kiebel:1999}. These estimators use the fact that for isotropic fields the LKCs are deterministic functions of the covariance matrix of first order partial derivatives of $f$ \citep[Corollary 11.7.3]{RFG:2007}. Thus, sophisticated discrete derivatives based on iid observations $f_1,...,f_N\sim f$ and averaging across $S$ gives accurate estimates of the LKCs. The simplifying stationary isotropic assumption, however, has been recently called into question in neuroimaging studies, claiming that it has led to too many false positive findings and lack of reproducibility \citep{Eklund:2016}.

A major complication in removing these strong assumption is that even knowing the functional form of a non-stationary covariance function exactly is generally not very helpful. While the LKCs can be written as integrals of covariances of partial derivatives of the field \citep[Thm 12.4.2]{RFG:2007}, these integrals are hard to evaluate numerically and analytically for arbitrary domains of dimension higher than 1.

Currently there are two approaches in the literature that relax this strong assumption on $f$. In \citet{Taylor:2007}, an estimator is introduced based on triangulating $S$ and warping the mesh of vertices using a diffeomorphism $\psi:S\rightarrow S$ such that the random field $f\circ\psi$ on the triangulated domain becomes locally isotropic. This is combined with smart computation of intrinsic volumes for simplicial complexes. This estimator, called {\it warping estimator} hereafter, is based on residuals, does not require stationarity and handles unknown mean and variance of $f$. Moreover, if the field $f$ is Gaussian, the warping estimator is unbiased if the triangulation of $S$ gets infinitely dense. An apparent downside of the warping estimator is its conceptual and computational complexity. The geometric calculations in the transformed space can be time-consuming for large triangulations and we show that this estimator is biased even for Gaussian observations on a finite triangulation.

A more recent approach, which inspired our work, is based on regression \citep{EC-regression2017}. Following the form of \eqref{eq:EEC}, they proposed to perform a linear regression of the average of empirically observed EC curves of a sample $f_1,...,f_N\sim f$ on the EC-densities. Choosing a set $u_1,...,u_P\in \R$ of exceedance levels then tranforms estimation of the LKCs into a general linear model with known covariates $\rho_d(u_p)$. The regression coefficients $\L_{1},...,\L_{D}$ are then estimated by weighted least squares. This approach, which they call {\it LKC regression}, has several problems. First, there are no clear guidelines on how to choose the locations $u_1,...,u_P$ on the real line, and the authors only compare heuristics for their placement. Second, the covariance function of the error vector of the regression needs to be estimated, which implies estimation of $\Cov\big[ \chi(A_{u_p}),\chi(A_{u_{p'}}) \big]$ for $p,p'=1,...,P$. There is theoretically not much known about this quantity for a general Gaussian random field. Thirdly, the authors give no theoretical analysis of their estimator, and fourthly, most importantly, their method relies heavily on observing Gaussian samples. This renders this estimator inapplicable in many practical situations.

\paragraph{Our proposed Hermite projection estimator}
Inspired by the regression approach, our proposed estimator solves all its problems. The key observation is that the EC densities \eqref{eq:EC-densities}, appropriately scaled, form an orthonormal system for a weighted $L^2$ space. Thus, the $d$-th LKC coefficient can be obtained by an appropriate orthogonal projection of the EEC curve onto the $d$-th EC density. We call this the {\it Hermite projection estimator} (HPE) of the LKCs. In Theorem \ref{thm:CritValRepresentation} we show that the HPE for Gaussian fields can be efficiently computed without numerically solving the indefinite integral that defines the projection. Under slightly stronger conditions than those required by the GKF, we also prove in Theorem \ref{thm:UnbiasedAndVariance} that the estimator is unbiased and has finite variance. These results allow us, by using the weak law of large numbers and the standard multivariate CLT, to draw power for estimation of the LKCs from a sample $f_1,...,f_N$ and show it is consistent and satisfies a CLT. A byproduct of the proof of Theorem \ref{thm:UnbiasedAndVariance} is the conjectured property that $\Cov[\chi(u),\chi(u')]$ decays faster than any polynomial in $u,u'$, see Corollary \ref{cor:CovDecay}.

The HPE in its simplest version requires mean zero variance one Gaussian fields to work properly, as illustrated in our simulations. The full solution of this problem comprises an unusual use of the Gaussian multiplier bootstrap, which allows non-Gaussian and even non-iid observations to estimate the LKCs of their Gaussian limiting field $f$. The Gaussian multiplier bootstrap allows to simulate from a Gaussian field, which conditioned on a sample $f_1,...,f_N$, is mean zero variance one with correlation function being the sample correlation. The simulated fields from this constructed Gaussian multiplier field can be fed into the HPE to obtain consistent LKC estimates, see Theorem \ref{thm:bHPE}. We call this estimator the {\it bootrapped HPE} (bHPE). In Theorem \ref{thm:bHPE} we show that it is consistent.

As proposed in \cite{EC-regression2017}, the EEC curve can be estimated by plugging an LKC estimate into the r.h.s of \eqref{eq:EEC}. If based on a single observed field, we refer to the corresponding EC curve as a \textit{smoothed EC curve}. This is illustrated for an isotropic field used in our simulations in Figure \ref{fig:GRF-isotropic}. Using the HPE and averaging the corresponding smoothed EC curves for multiple realizations leads to a linear and smooth parametric estimator of the EEC curve, which we also call HPE. This estimator satisfies a functional central limit theorem (fCLT), see Theorem \ref{thm:ECcurveCLT}, from which confidence bands can be obtained. Alternatively, a nonparametric estimate may be obtained by pointwise averaging of the observed EC curves, compare Figure \ref{fig:GRF-isotropic}(middle). Simulations show that the HPE has a lower variance, which can be explained by the fact that the HPE is an orthogonal projection of the nonparametric estimator onto the EC densities.

Additionally, for applications in FWER inference, we show in Theorem \ref{thm:CLT-threshold} that our EEC estimator leads to a consistent estimator of the detection threshold and we derive confidence intervals.
\vspace*{-0.3cm}

\begin{figure}[H]
\centering
		\includegraphics[trim=0 0 0 0,clip,width=2.1in]{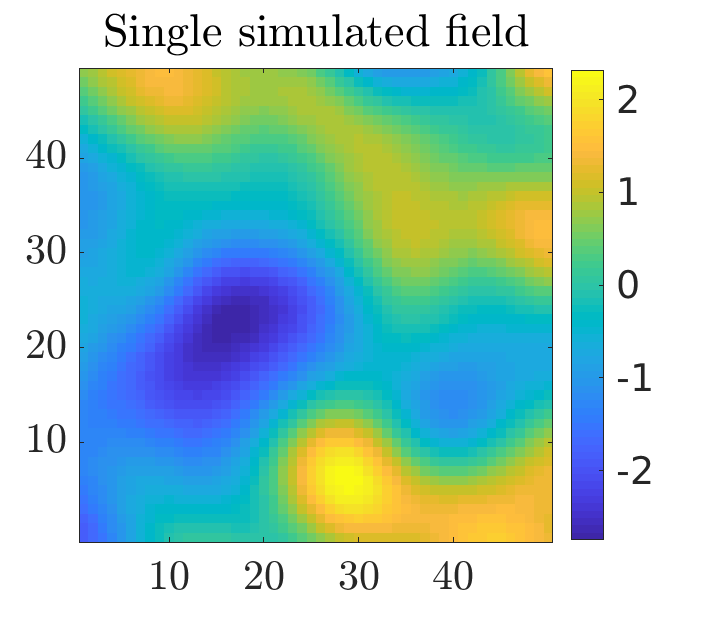}
		\includegraphics[ trim=0 0 0 0, clip, width=2.1in ]{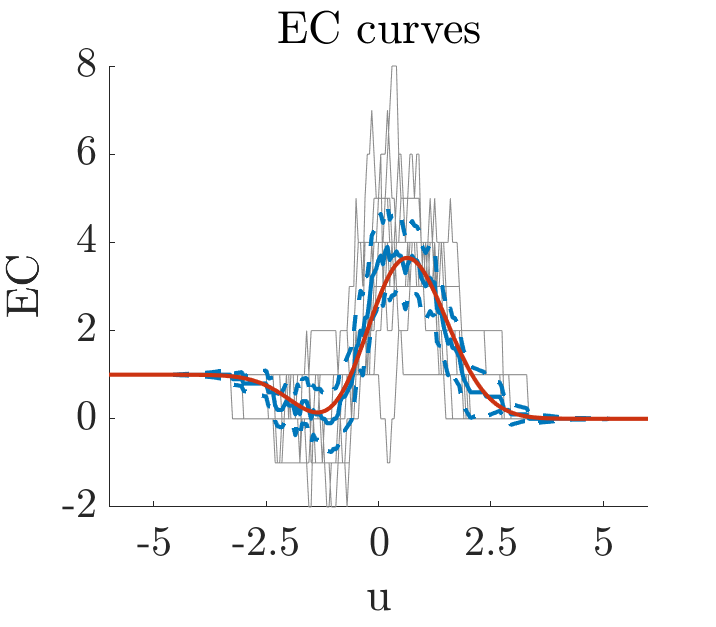}
		\includegraphics[ trim=0 0 0 0, clip, width=2.1in ]{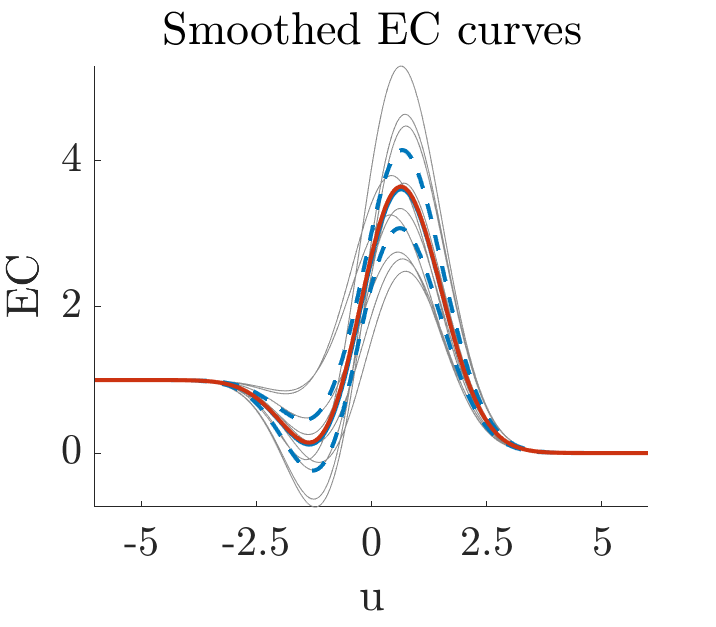}
		
		\vspace*{-0.7cm}
		\caption{(left) Single realization of the isotropic Gaussian random field \eqref{eq:isotropic-2D}. (middle) EC curves (gray) for $N=10$ realizations of the field and their average (blue). (right) Corresponding smoothed EC curves (gray) and their pointwise average (blue). Dashed blue lines are pointwise 95\% confidence bands for the true EEC curve (red).
		\label{fig:GRF-isotropic}}
\end{figure}

		\vspace*{-0.7cm}
\paragraph{Simulations and data applications}
The simulations in Section \ref{sec:simulations} studies the theoretical properties and the finite sample performance of the HPE and bHPE for Gaussian, non-Gaussian and non-stationary fields on a 2D rectangular domain. Performances are compared to a method (IsotE) from \citep{Kiebel:1999} taylored to stationary isotropic fields and the warping-to-isotropy transformation (WarpE) of \cite{Taylor:2007}, which applies to non-isotropic fields. Note that the HPE can be seen as a continuous version of the LKC regression estimator and thus, has similar properties. Therefore we do not show results of the EC regression method. In all cases the bHPE gives comparable or better results than its competitors.

In Section \ref{sec:data} we estimate EECs of the cosmic microwave background (CMB) radiation field on a complex, non-trivial subset of the 2-sphere using cosmological simulations \citep{Planck:2016-ffp8} in order to compare the physical model with actual observed CMB data from the Planck satellite \citep{Planck:2016-NILC}. A second data application for FWER inference on voxelwise activation in a single-subject fMRI study \citep{Moran:2012} is available in the supplementary material.
Matlab code for reproducing  the simulation results and data analysis is available under \url{https://github.com/ftelschow/HPE}.

\section{Estimation of the LKCs for Gaussian fields}
\label{sec:Gaussestimator}
\paragraph{LKCs as projections of the EC curve}
\label{sec:functional}
Consider the Hilbert space with inner product
\begin{equation}
\label{eq:inner-product}
\langle g, h \rangle = \int_{-\infty}^\infty g(u) h(u) e^{u^2/2} \,du,
\end{equation}
consisting of all functions $g$, $h$ such that $|g|^2 = \langle g, g \rangle < \infty$ and $|h|^2 = \langle h, h \rangle < \infty$. The key observation for estimation of LKCs is that the EC densities \eqref{eq:EC-densities} are orthogonal, i.e.,
\begin{equation}
\label{eq:orthogonality-rho}
\langle \rho_d, \rho_{d'} \rangle
= \int_{-\infty}^\infty \rho_d(u) \rho_{d'}(u) e^{u^2/2} \,du
= (2\pi)^{-(d+1/2)} (d-1)! \delta_{d{d'}}\,,
\end{equation}
where $\delta_{d{d'}}$ is the Kronecker delta.
 This follows from the orthogonality of the Hermite polynomials in the $L^2$ space with weights $e^{-u^2/2}$, i.e.,
\begin{equation}
\label{eq:orthogonality-H}
\int_{-\infty}^\infty H_{d-1}(u) H_{{d'}-1}(u) e^{-u^2/2} \,du = \sqrt{2\pi} (d-1)! \delta_{d{d'}}, \qquad d,{d'} = 1,2,\ldots,\,.
\end{equation}

Thus, \eqref{eq:EEC} implies that the LKCs can be obtained from the $\EEC$ as projection coefficients,
\begin{equation}
\label{eq:L-recover}
\L_d = \frac{ \left\langle  \EEC^{\circ}, \rho_d \right\rangle}{|\rho_d|^2} = \tfrac{(2\pi)^{d/2}}{(d-1)!} \int_{-\infty}^\infty H_{d-1}(u) \EEC^{\circ}(u) \,du\,,
\end{equation}
where $\EEC^{\circ} = \EEC - \L_0 \Phi^+$  is the ``pinned" EEC curve. It is pinned in the sense that it tends to 0 for both small and large $u$, since $\lim_{u\to \infty} \EEC(u)=0$ and $\lim_{u\to -\infty} \EEC(u)=\L_0$. The operation \eqref{eq:L-recover} can be seen as a linear functional on the Hilbert space 
\begin{equation}\label{eq:projection}
\H_d\{ g \} = \frac{ \left\langle  g, \rho_d \right\rangle}{|\rho_d|^2} =  \frac{(2\pi)^{d/2}}{(d-1)!} \int_{-\infty}^\infty H_{d-1}(u) g(u) \,du, \qquad d = 1,2,\ldots\,,
\end{equation}
which we call the $d$-th Hermite projector.

\paragraph{Estimation of the LKCs from a single observation}
\label{sec:observed-LKC}
Because $\L_0$ is the EC of the domain $S$ of the Gaussian random field $f$, it is known and need not be estimated. Given a realization of $f$ and its empirical EC curve $u\mapsto \chi_f( u )$, which is formed by the EC of the excursion sets $A_f(u)$ of $f$ above $u$, we define the corresponding ``pinned" EC curve as
\begin{equation}
	\label{eq:EC-pinned}
	\chi^{\circ}_f:~ \R\rightarrow \R \,,~~u\mapsto \chi_f( u ) - \L_0 \Phi^+(u)\,,
\end{equation}
which satisfies $\lim_{u\to \infty} \chi_f(u)=0$ and $\lim_{u\to -\infty} \chi_f(u)=\L_0$ almost surely. Applying the Hermite projector \eqref{eq:projection} yields the estimator
\begin{equation}
	\label{eq:observed-LKC}
	\hat{\L}_{d} = \H_d\{\chi^{\circ}_f\} = \tfrac{(2\pi)^{d/2}}{(d-1)!} \int_{-\infty}^\infty H_{d-1}(u) \chi^{\circ}_f(u) \,du\,,
\end{equation}
which we call the {\it Hermite projection estimator} (HPE) of the LKC $\L_d$.
Note that the integral is well defined, since $\chi^{\circ}_f$ is exponentially decaying outside the interval $[u_0, u_M]$ and is bounded on it. Here $u_0,u_M$ denote the values of the global minima and maxima, respectively, of the realization $f$. Comparing with \eqref{eq:L-recover}, the estimator $\hat{\L}_{d}$ plays the role of the ``observed" LKC of order $d$ of the field $f$.  In some sense the HPE can be seen as a continuous version of Adler's LKC regression, compare Appendix \ref{appendix:lin-reg-view}.

\paragraph{Properties of the Hermite projection estimator}\label{sec:prop-LKC-estimator}

Heuristically, from equations \eqref{eq:L-recover} and \eqref{eq:observed-LKC} by interchanging integration and expectation we obtain that $\hat{\L}_{d}$ is unbiased. Moreover, let $\hat{\bm \L} = (\hat{\L}_{1},\ldots,\hat{\L}_{D})^{\tt T}$ be the vector of observed LKCs and denote its covariance matrix by $\bSigma = \Cov\left[\hat{\bm \L}\right]$. Again changing order of integration and expectation yields that the $(d,{d'})$-entry of the covariance matrix $\bSigma$ can be expressed as
\begin{equation}\label{eq:cov-L}
\sigma_{d{d'}} = \Cov\left[\hat{\L}_{d}, \hat{\L}_{{d'}}\right]
= \tfrac{(2\pi)^{d/2} (2\pi)^{{d'}/2}}{(d-1)! ({d'}-1)!} \iint H_{d-1}(u) \, H_{{d'}-1}(v) \, \Cov\left[\chi_f( u ), \chi_f( v )\right] \,du\,dv <\infty.
\end{equation}
using Eq. \eqref{eq:observed-LKC}. To rigorously prove these statements we require the following assumptions.
\begin{enumerate}[leftmargin=1.4cm]
    \item[\textbf{(G1)}]  $f$ is a mean zero, variance one Gaussian field with almost surely $\mathcal{C}^2$-sample paths.
    \item[\textbf{(G2)}] The  distribution of $\Big(\tfrac{\partial f}{\partial s_d}(s), \tfrac{\partial^2 f}{\partial s_d\partial s_{d'}}(s)\Big)$ is nondegenerate for all $s\in S$ and $d,d'=1,...,D$.
    \item[\textbf{(G3)}] There is $\epsilon >0$ such that for all $d,d'=1,...,D$ and all $\vert s-s'\vert <\epsilon$ it holds that $$\E\!\left[ \Big( \tfrac{\partial^2 f}{\partial s_d\partial s_{d'}}(s) -\tfrac{\partial^2 f}{\partial s_d\partial s_{d'}}(s') \Big)^2 \right] \leq K \big\vert \log\Vert s-s' \Vert \big\vert^{-(1+\gamma)}. $$ Here $K>0$ and $\gamma>0$ are finite constants.
    \item[\textbf{(G4)}] Let $N_f$ be the number of critical points of $f$ and $\varepsilon>0$.\\
        $~~~~~\text{ a) }~\E\big[N_f^{1+\varepsilon}\big] < \infty~ ~ ~ ~ ~ ~ ~ ~\text{ b) }~ \E\big[N_f^{2+\varepsilon}\big] < \infty$
\end{enumerate}

\begin{remark}
\textbf{(G1)}-\textbf{(G3)} are requirements for the validity of the GKF, see \citet{RFG:2007}. In particular, they imply that the paths of $f$ are almost surely Morse functions, \cite[Chapter 11.3]{RFG:2007}. \textbf{(G3)} is satisfied for any Gaussian field $f$ having almost surely $\mathcal{C}^3-$sample paths \citep{Telschow:2019}. \textbf{(G4)} is sufficient to proof unbiasedness (a) and finiteness of the variance (b) of our estimator. Precise conditions to satisfy these conditions are an active research topic. For $D=1$ the almost surely $\mathcal{C}^4/\mathcal{C}^7$-sample paths are sufficient to establish the finiteness of the moments of the number of critical points in \textbf{(G4a/b)}, cf. \cite[Theorem 3.6]{Azais:2009}. For $D>1$ the situation is more complicated. However, at least for stationary fields weak sufficient conditions to satisfy \textbf{(G4a)} have been recently given by \citep{Estrade:2016}, which are satisfied for Gaussian fields with $\mathcal{C}^3$ sample paths.
\end{remark}

Note that, if we order the critical values of $f$ by height, then the empirical EC curve $\chi_f(u)$ is constant between consecutive critical levels. Therefore the estimator \eqref{eq:L-hat} can be equivalently written as a polynomial function of the critical levels.
\begin{thm}\label{thm:CritValRepresentation}
 Assume $f$ is Gaussian and satisfies \textbf{(G1)}-\textbf{(G4)}. Since the paths of $f$ are almost surely Morse functions, we only have finitely many critical values, which we order and denote with $u_0<...<u_M$. Thus, almost surely
\begin{equation}\label{eq:EC-crit-val-repr}
 \chi_f(u) = \L_0\mathds{1}_{(-\infty,u_{0}]}(u) + \sum_{m=1}^{M} a_m \mathds{1}_{(u_{m-1},u_{m}]}(u)
\end{equation}
with random $a_m=\chi_f(u_{m})\in\mathbb{Z}$ and hence
\begin{align*}
 \hat\L_{d} &= \frac{(2\pi)^{d/2}}{d!}\sum_{m=0}^{M} (a_m-a_{m+1}) H_d(u_{m})\,,~~ \text{ with } a_0=\L_0 \text{ and } a_{M+1}=0\,.
\end{align*}
\end{thm}
\begin{remark}
This representation of $\hat\L_{d}$ is useful for numerical implementation, since it does not involve the indefinite integral \eqref{eq:observed-LKC} anymore. Note that the critical values $u_{m}$ can be efficiently estimated together with the coefficients $a_m$, compare Appendix \ref{appendix:ECcomputation}. In fact, note that $a_m-a_{m-1}\in\{\pm 1\}$ for $m=1,...,M$, since crossing a critical value changes the EC of the excursion set only by $\pm 1$, see \citet[Theorem 9.3.2.]{RFG:2007}.
\end{remark}

The next theorem states that the LKC estimators are unbiased and that the covariances are finite under proper assumptions. The strategy of the proofs is identical. We show using Young's and Borel TIS inequality that Fubini's theorem is applicable. Especially, this implies that \eqref{eq:cov-L} holds. Hence $\Cov[\chi_f(u)\chi_f(u')]$ must decay fast enough, which is stated in the corresponding corollary and gives a partial answer to the conjecture that the variance $\Var[\chi_f(u)]$ is decaying exponentially in $u$.
\begin{thm}\label{thm:UnbiasedAndVariance}
	\label{thm:LKC-unbiased} Assume that $f$ is Gaussian and satisfies \textbf{(G1)}-\textbf{(G4)}.
	\begin{enumerate}
	 \item[(i)] If $(G4a)$, then the HPEs are unbiased, i.e., $\E\big[ \hat\L_{d} \big] = \L_{d}$.
	 \item[(ii)] If $(G4b)$, then the covariances $\sigma_{dd'} = \Cov\left[\hat{\L}_{d}, \hat{\L}_{d'}\right]$ are finite for all $d,d'=1,...,N$. 
	\end{enumerate}
\end{thm}

\begin{corr}\label{cor:CovDecay}
    Under the assumptions of Theorem \ref{thm:LKC-unbiased}\,b) we have that
    $\Cov\big[\chi_f(u)\chi_f(u') \big]$ decays faster than any polynomial in $u,u'$ for $u,u'\rightarrow \pm\infty$. 
\end{corr}

\paragraph{Estimation of the LKCs from repeated observations}
\label{sec:repeated-obs}

Let $f_1,...,f_N\sim f$ be iid random fields over the domain $S$ with unknown covariance function satisfying \textbf{(G1)}-\textbf{(G3)} and let ${\bm \L} = ({\L}_{1},\ldots,{\L}_{D})^{\tt T}$ be the vector of true LKCs. Then each EC curve $\chi_{f_n}$ yields a vector of LKC estimates $\hat{\bm \L}_n = (\hat{\L}_{1n},\ldots,\hat{\L}_{Dn})^{\tt T}$. The average estimator is defined by
\begin{equation}
\label{eq:L-hat}
\hat{\bm \L}^{(N)} = (\hat{\L}_1^{(N)},\ldots,\hat{\L}_D^{(N)})^{\tt T} = \frac{1}{N}\sum_{n=1}^N \hat{\bm \L}_n\,.
\end{equation}

The next result shows unbiasedness, consistency and a CLT for the this estimator. It is an immediate consequence of Theorem \ref{thm:UnbiasedAndVariance} by applying the strong law of large numbers (SLLN) and the standard multivariate central limit theorem (CLT).
\begin{corr}\label{cor:repLKC-unbiased}
	Under the assumptions of Theorem \ref{thm:UnbiasedAndVariance}:
	\begin{enumerate}
		\item[(i)] Assume\textbf{(G4a)}. Then $\hat{\bm \L}^{(N)}$ is unbiased and $\hat{\bm \L}^{(N)}\xrightarrow{~a.s.~} \bm{\L},$ as $N \to \infty$.
		\item[(ii)] Assume \textbf{(G4b)}. Then
					$\sqrt{N} \left(\hat{\bm \L}^{(N)} - {\bm \L} \right) \xrightarrow{~D~}  N(0,\bSigma)$,
		where $\bSigma$ is given by \eqref{eq:cov-L}.
	\end{enumerate}
\end{corr}

With repeated observations, the covariance matrix $\bSigma$ can be estimated unbiasedly and consistently from the data via
\begin{equation}
\label{eq:cov-L-hat}
\hat{\bSigma}^{(N)} = \frac{1}{N-1} \sum_{n=1}^N \left[\hat{\bm \L}_n - \hat{\bm \L}^{(N)}\right]\left[\hat{\bm \L}_n - \hat{\bm \L}^{(N)}\right]^{\tt T}.
\end{equation}

\section{LKC estimation from non-Gaussian fields}\label{sec:LKC-estim-asym}

The previously developed theory relies heavily on the assumption that $f$ is a mean zero and variance one Gaussian field. Usually, these assumptions are not satisfied in applications. Therefore the purpose of this section is to remove them, including that the observed fields are independent and identically distributed.

The first observation is that LKCs are rather a property of the correlation of a field $f$ than a property of the field itself. They solely depend on the Riemannian metric on $S$ induced by $f$, which is determined by the correlation of its gradient field, compare \cite[Chapter 7 and 12]{RFG:2007}. Secondly, inferential tools often rely on a fCLT, as for example in neuroimaging, see Section \ref{sec:linear-models}.
Thus, often only the LKCs of a limiting mean zero variance one Gaussian field $G$ with correlation $\mathfrak{r}$ are of interest.

\paragraph{Bootstrapped Hermite projection estimator}\label{sec:bHPEdef}
Assume we are interested in the LKCs of a limiting Gaussian field $G$ satisfying \textbf{(G1)}-\textbf{(G4)}. In order to generalize the HPE, let $R_1,...,R_N$ be random fields constructed from observations satisfying
\begin{enumerate}[leftmargin=1.4cm]
    \item[\textbf{(R1)}] $\sum_{n=1}^N R_n(s) = 0$ and $\sum_{n=1}^N R_n^2(s) = 1$ for all $s\in S$.
    \item[\textbf{(R2)}] $\hat{\mathfrak{r}}\vphantom{r}^{(N)}(s,s')=N^{-1}\sum_{n=1}^N R_n(s)R_n(s')$ and its partial derivatives up to order 2 converge uniformly almost surely to ${\mathfrak{r}}$ and its partial derivatives, respectively.
\end{enumerate}
We call $R_1,...,R_N$ \textit{standardized residuals}. The name stems from the following example. Suppose $f_1,...,f_N\sim f$ is an iid sample of random fields satisfying a fCLT, i.e., 
\begin{equation}
	\frac{1}{\sqrt{N}}\sum_{n=1}^N \frac{ f_N-\E[f] }{\sqrt{\Var[f]}} \xrightarrow{~D~}G
\end{equation}
as $N\rightarrow\infty$. Further assume that the mean and variance are unknown. Then standardized residuals satisfying \textbf{(R1)}, and under appropriate conditions on $f$ also \textbf{(R2)}, are given by
\begin{equation}
 R_n = \frac{f_n - \bar f}{\sqrt{\sum_{n=1}^N \big(f_n-\bar f\big)^2}}\,.
\end{equation}
Note that, even if $f$ is Gaussian, these standardized residuals are not. Thus, they cannot immediately be used to estimate the LKCs with the HPE.

The key is that LKCs depend only on the correlation ${\mathfrak{r}}$. Therefore, using standardized residuals, we construct a Gaussian field approximating this unknown correlation.
\begin{defn}[Gaussian Multiplier Field]\label{def:GMF}
    Given a set of standardized residuals $R_1,..., R_N$ satisfying \textbf{(R1)} and \textbf{(R2)}, we define the \textit{Gaussian multiplier field} (GMF) by
    \begin{equation}\label{eq:AGMB}
            R^{(N)}_{\mathbf{g}} = \sum_{n=1}^N g_n R_n\,,
    \end{equation}
    where $\mathbf{g}=(g_1,...,g_N)\sim N(0,I_{N\times N})$.
\end{defn}

The GMF conditioned on $R_1,..., R_N$ is a Gaussian field and it can immediately be verified that the GMF satisfies
\begin{equation}
	\E\Big[ R^{(N)}_{\mathbf{g}}~\mid~R_1,...,R_N ~\Big] = 0 ~ ~\text{ and }~ ~
	\E\Big[ R^{(N)}_{\mathbf{g}}(s) R^{(N)}_{\mathbf{g}}(s')~\mid~R_1,...,R_N ~\Big] = \hat{\mathfrak{r}}\vphantom{r}^{(N)}(s,s')
\end{equation}
for all $s,s'\in S$. Thus, for fixed observed residuals $R_1,..., R_N$, we apply the HPE \eqref{eq:observed-LKC} to an iid bootstrap sample $R^{(N)}_{\mathbf{g}_1},...,R^{(N)}_{\mathbf{g}_M}\sim R^{(N)}_{\mathbf{g}}$ and obtain an estimator 
\begin{equation}\label{eq:ECproj-boots-estim}
    \bm{\hat{\L}}_B^{(N)} = \lim_{M\rightarrow\infty}\frac{1}{M}\sum_{m=1}^M \bm{\hat{\L}}\Big( R^{(N)}_{\mathbf{g}_m}  \Big).
\end{equation}
of the LKCs of the mean zero variance one Gaussian field having correlation structre ${\mathfrak{r}}$. We call this estimator the {\it bootstrap Hermite projection estimator} (bHPE).

\paragraph{Properties of the bootstrapped Hermite projection estimator}\label{subscn:bHPE-prop}
Since the bHPE is defined as a limit, we need to show that this limit exists. This is established using Theorem \ref{thm:UnbiasedAndVariance}, since $R^{(N)}_{\mathbf{g}}$ conditioned on the residuals is a Gaussian field. We also prove that the bHPE is consistent. Let $\L(\mathfrak{r})$ denote the vector of LKCs of a field $G$ having correlation $\mathfrak{r}$.

\begin{thm}\label{thm:bHPE} Assume that for almost all $R_1,..,R_N$ and almost all $N>N'$ the Gaussian field $R^{(N)}_{\mathbf{g}}$ satisfies \textbf{(G1)}-\textbf{(G4a)} and that \textbf{(R1)} and \textbf{(R2)} are satisfied. Then
    \begin{enumerate}
        \item[(i)]  $\bm{\hat{\L}}_B^{(N)} = \bm{{\L}}\big(\hat{\mathfrak{r}}\vphantom{r}^{(N)}\big)$ almost surely.
        \item[(ii)] $\bm{\hat{\L}}_B^{(N)} \xrightarrow{~a.s.~} \bm{{\L}} (\mathfrak{r} )$ as $N\rightarrow\infty$.
    \end{enumerate}
\end{thm}
 Part $(i)$ shows that, for finite $N$, the bHPE is equal to the LKCs of a Gaussian field having the covariance structure of the sample correlation $\hat{\mathfrak{r}}\vphantom{r}^{(N)}$. Therefore it is a continuous version of the discrete quantity computed in the warping method \cite{Taylor:2007} and hence both estimators will have a similar variance, see Section \ref{sec:simulations}.

\section{Estimation of the EEC and Applications}
\paragraph{Plug-in estimation of the EEC}
\label{sec:parametric}

More generally, assume temporarily for this section that $\hat{\bm{\L}}\vphantom{L}^{(N)}=\big( \hat{\L}^{(N)}_1, ..., \hat{\L}^{(N)}_D \big)$ is an arbitrary estimator of the LKCs $\bm{{\L}}$. Given these estimates \citet{EC-regression2017} suggest the use of a plug-in estimator into the GKF
\begin{equation}
	\label{eq:EEC-hat-HPE}
	\widehat{\EEC}\vphantom{EEC}^{(N)}\!(u) = \L_0 \Phi^+(u) + \sum_{d=1}^D \hat{\L}^{(N)}_d \rho_d(u)
\end{equation}
to obtain a smooth estimate of the EEC curve. By linearity, we obtain that this estimator is unbiased for the EEC if $\hat{\bm{\L}}\vphantom{L}^{(N)}$ is an unbiased estimator. Moreover, since
\begin{equation}
	\widehat{\EEC}\vphantom{EEC}^{(N)}\!(u) - \EEC(u)  = \sum_{d=1}^D \Big(\hat{\L}^{(N)}_d - \L_{d}\Big) \rho_d(u)\,,
\end{equation}
the covariance function of the plug-in estimator can be expressed as
\begin{equation}
\label{eq:cov-chi-hat}
C(u,v) = \sum_{d=1}^D \sum_{d'=1}^D \sigma_{dd'}^{(N)} \rho_d(u) \rho_{d'}(v),
\end{equation}
where $\sigma_{dd'}^{(N)}$ is the $(d,d')$ entry of the covariance matrix $\bSigma^{(N)}$ of the vector $\hat{\bm{\L}}\vphantom{L}^{(N)}$. Note that the latter only is well-defined if $\bSigma^{(N)}$ exists. Similarily, consistency and a fCLT can be derived from corresponding properties of the estimator $\hat{\bm{\L}}\vphantom{L}^{(N)}$.

\begin{thm}\label{thm:ECcurveCLT}
Under the assumptions \textbf{(G1)-(G4)} the smooth plug-in estimate of the EEC curve has the following properties:
\begin{enumerate}
\item[(i)] Assume that $\hat{\bm{\L}}\vphantom{L}^{(N)}$ is consistent. Then the smooth plug-in estimate $\widehat{\EEC}\vphantom{EEC}^{(N)}$ and all its derivatives $\tfrac{d^k}{du^k}\widehat{\EEC}\vphantom{EEC}^{(N)}$ for $k\geq1$ are uniformly consistent estimators of $\EEC$ and its derivatives $\tfrac{d^k}{du^k}\EEC$, respectively.

\item[(ii)] Assume that $\sqrt{N}\Big(\hat{\bm{\L}}\vphantom{L}^{(N)}-\bm{\L}\Big)\xrightarrow{~D~} N(0,\bSigma)$. Then the smooth plug-in estimate $\widehat{\EEC}\vphantom{EEC}^{(N)}$ satisfies a fCLT such that
\begin{equation}
\label{eq:EEC-hat-CLT}
\sqrt{N} \left[\widehat{\EEC}\vphantom{EEC}^{(N)}\!(u) - \EEC(u)\right] \xrightarrow{~D~} G(u) = \sum_{d=1}^D Z_d \rho_{d}(u)
\end{equation}
weakly in $C(\R)$ as $N\rightarrow\infty$, where $(Z_1,\ldots,Z_D)$ is a multivariate Gaussian random variable with mean zero and covariance $\bSigma$.
\end{enumerate}
\end{thm}
\begin{remark}
	\citet{Taylor:2006} proved a Gaussian Kinematic formula for Gaussian related fields. These fields can be written as $F(f_1,...,f_K)$, where $F:\R^K\rightarrow\R$ is twice continuously differentiable and $f_1,...,f_K\sim f$ are iid mean zero, variance one Gaussian fields satisfying \textbf{(G1)}-\textbf{(G4)}. Examples are pointwise $\chi^2$- and $t$-distributed fields. The above theorem immediately generalizes to plug-in estimators for the $\EEC$ obtained from such GKFs. 
\end{remark}

\paragraph{Hermite Projection Estimator of the EEC}
\label{sec:HPE-EEC}

Consider the EEC plug-in estimator \eqref{eq:EEC-hat-HPE} using the HPE of the LKCs from Section \ref{sec:repeated-obs}. With some abuse of nomenclature, we call this the \textit{ Hermite projection estimator} (HPE) of the EEC.
Theorem \ref{thm:ECcurveCLT} can be applied to this estimator, since the assumptions are satisfied by Corollary \ref{cor:repLKC-unbiased}.
\begin{corr}\label{cor:HPEsatisfiesEEC}
    Under \textbf{(G1)}-\textbf{(G4)} the result of Theorem \ref{thm:ECcurveCLT} is true for the HPE of the EEC. 
\end{corr}

We call this EEC estimator HPE because it is in itself a projection with respect to the Hilbert space structure introduced in Section \ref{sec:functional} onto the subspace spanned by the EC densities $\rho_1,...,\rho_D$. Let $f_1,...,f_N\sim f$ be iid mean zero variance one Gaussian fields. Each HPE $\big( \hat{\L}_{1n},...,\hat{\L}_{Dn} \big)$ of LKCs from a single observation $f_n$ allows for a smooth estimate of the EC curves as a linear combination of the EC densities via
\begin{equation}
	\label{eq:EC-smoothing}
	\hat{\chi}_{f_n}(u) = \L_0 \Phi^+(u) + \sum_{d=1}^D \hat{\L}_{dn} \rho_d(u)\,.
\end{equation}
In fact, by linearity we can write the plugin EEC estimator using the HPE of the LKCs as the average of these smoothed EC curves, i.e.,
\begin{equation}
	\label{eq:EEC-hat-avg}
	\widehat{\EEC}\vphantom{EEC}^{(N)}\!(u) = \L_0 \Phi^+(u) + \sum_{d=1}^D \hat{\L}_{d}^{(N)} \rho_{d}(u) = \frac{1}{N}\sum_{n=1}^N \hat{\chi}_{f_n}(u)\,.
\end{equation}

At this point one may consider the alternative more straightforward estimation of the EEC curve as the sample average of the observed EC curves,
\begin{equation}
	\label{eq:ECavg}
	\bar{\chi}^{(N)}(u) = \frac{1}{N}\sum_{n=1}^N \chi_f(u)\,.
\end{equation}
Since these curves are iid and have finite variance under \textbf{(G4b)}, \eqref{eq:ECavg} is an unbiased and consistent estimator of the EEC satisfying pointwise a CLT (proving a fCLT as in \eqref{eq:EEC-hat-CLT} is tedious, since the covariance function $(u,v)\mapsto \Cov\big[ \chi_f(u), \chi_{f}(v) \big]$ is not well studied).

It turns out that the HPE of the EEC is a smoothed version of the raw average in \eqref{eq:ECavg}. Using linearity of the Hermite projection \eqref{eq:projection}, the HPE \eqref{eq:L-hat} of LKCs can be equivalently obtained from the average EC curve \eqref{eq:ECavg} as
    \begin{equation}
    \label{eq:L-hat-equiv}
    \hat{\L}_d^{(N)} = \H_d\Big\{ \bar{\chi}^{(N)}(u) - \L_0 \Phi^+(u) \Big\}.
    \end{equation}
    Thus, using the HPE and plugin it into \eqref{eq:EEC-hat-avg} and using \eqref{eq:projection} yields
    \begin{equation}
		\label{eq:EEC-hat}
		\widehat{\EEC}\vphantom{EEC}^{(N)}\!(u) = \L_0 \Phi^+(u) + \sum_{d=1}^D \frac{ \left\langle \bar{\chi}^{(N)}(u) - \L_0 \Phi^+(u) , \rho_d \right\rangle}{|\rho_d|^2} \rho_d(u)\,.
	\end{equation}
In fact, this is the orthogonal projection of $\bar{\chi}_f(u)$ onto the vector space spanned by the EC densities according to the inner product \eqref{eq:inner-product}, hence the name HPE.

The fCLT \eqref{eq:EEC-hat-CLT} for $\widehat{\EEC}\vphantom{EEC}^{(N)}$ allows constructing confidence regions for the EEC curve. In particular, asymptotic two-sided $1-\alpha$  pointwise confidence bands can be built as
\begin{equation}\label{eq:EEC-hat-conf}
\widehat{\EEC}\vphantom{EEC}^{(N)}\!(u) \pm z_{1-\alpha/2} \sqrt{\hat{C}(u,u)/N},
\end{equation}
where the estimate of the variance function $\hat{C}(u,u)$ is obtained as a plug-in estimator substituting the sample covariance $\hat{\bSigma}^{(N)}$ given into \eqref{eq:cov-L-hat} in \eqref{eq:cov-chi-hat} and $z_{1-\alpha/2}$ is the $1-\alpha/2$ quantile of the standard normal distribution.
Moreover, using the GKF or the multiplier-t bootstrap for the normalized limiting process $G(u)$ in \eqref{eq:EEC-hat-CLT}, it is possible to provide asymptotic simultaneous confidence bands for $\EEC(u)$ for any compact set $U\subset\R$ using the construction from \cite{Telschow:2019}.

In our simulations in Section \ref{sec:simulations} we will show that the confidence bands \eqref{eq:EEC-hat-conf} are tighter and have better coverage than the pointwise confidence intervals
\begin{equation}\label{eq:ECavg-conf}
	\bar{\chi}^{(N)}(u) \pm z_{1-\alpha/2} \sqrt{\widehat{\Var}[\chi_f(u)]/N}\,,
\end{equation}
where $\widehat{\Var}[\chi_f(u)]$ is the sample variance of $\chi_{f_1}(u),\ldots,\chi_{f_N}(u)$, which are based on the pointwise CLT for the sample average $\bar{\chi}^{(N)}$.

\paragraph{Inference based on the estimated EEC}
\label{sec:threshold}

The estimated EEC curve is used as a tool for statistical inference, for example, by finding a threshold $u$ such that $\widehat{\EEC}(u)$ is less or equal to a pre-specified value $\alpha$. This was the pioneering approach taken by Keith Worsley in the analysis of brain images, see for example \citep{Worsley:1996,Worsley:2004}. Especially, he used the EEC curve to bound the FWER of voxelwise given test statistics on an image. In order to achieve this he used the so called \textit{Euler characteristic heuristic}, which states that for Gaussian related fields $T$ and high enough threshold $u$ the maximum of $T$ being larger than $u$ can be well approximated by the EEC of the excursions set above $u$, i.e.,
\begin{equation}
	\alpha = \EEC(u_\alpha) \approx \P\left\{\sup_{s\in S} T(s) > u_\alpha\right\}\,.
\end{equation}
This approximation typically works well with $\alpha \leq 0.05$. If $T$ is a mean zero variance one Gaussian field, this heuristic has been theoretically validated in \citet{Taylor:2005}.

Another interpretation of thresholding the EEC curve is the following. For high but somewhat lower thresholds than for FWER control, the EEC curve counts the expected number of connected components in $A_n(u)$. Thus the threshold $u_\alpha$ controls the expected number of connected components, or cluster error rate (CER), to be below $\alpha$. Although uncommon, this approach has been used in neuroimaging as well with $\alpha = 1$ \citep{Bullmore:1999}, meaning the excursion set contains on average one false connected component.

Suppose now that the EEC curve is estimated via \eqref{eq:EEC-hat} and a significance threshold $\hat{u}_\alpha^{(N)}$ is found that satisfies
$\widehat{\EEC}\vphantom{EEC}^{(N)}\!\left(\hat{u}_\alpha^{(N)}\right) = \alpha$.
We wish to know how variable is this threshold.
The following Theorem gives the asymptotic properties of the estimated threshold. The results are valid for any EEC estimator for which the results of Theorem \ref{thm:ECcurveCLT} hold, in particular the HPE.
 Let $\EEC'(u)$ denote the derivative of the EEC curve, i.e.,
		\[
		\EEC'(u) = \sum_{d=0}^D \L_{d} \rho_{d}'(u) = \sqrt{2\pi} \sum_{d=0}^D \L_{d} \rho_{d+1}(u).
		\]
\begin{thm}
	\label{thm:CLT-threshold}
	Assume that the results of Theorem \ref{thm:ECcurveCLT} are satisfied. Then the random threshold $\hat{u}_\alpha^{(N)}$ has the following properties:
	\begin{enumerate}
		\item[(i)] $\hat{u}_\alpha^{(N)}$ is a consistent estimator of $u_\alpha$.
		\item[(ii)] Let $\alpha$ be small enough such that $\EEC'(u_\alpha) \ne 0$. Then $\sqrt{N}\left(\hat{u}_\alpha^{(N)} - u_\alpha\right)$ converges in distribution to $G(u_0)/\EEC'(u_\alpha)$,
		where $G(u)$ is the Gaussian process in \eqref{eq:EEC-hat-CLT}.
	\end{enumerate}
\end{thm}

A consequence of Theorem \ref{thm:CLT-threshold} is that, for large $N$, $\hat{u}_\alpha$ is approximately Gaussian with mean $u_\alpha$ and variance
\begin{equation}\label{eq:var-u}
\Var\left[\hat{u}^{(N)}_\alpha\right] = \frac{C(u_\alpha,u_\alpha)}{N [\EEC'(u_\alpha)]^2},
\end{equation}
where $C(u,u)$ is given by \eqref{eq:cov-chi-hat}.
In practice, we estimate this variance by substituting estimators for $u_\alpha$, $C(u,u)$ and the LKCs as described above.

\vspace{-0.4cm}

\section{Simulation studies}\label{sec:simulations}

\paragraph{Design of the simulations}\label{sec:SimulationSetup}
 
We compare the Hermite projection estimator (HPE) given by \eqref{eq:L-hat}, the bootstrapped Hermite projection estimator (bHPE) from \eqref{eq:ECproj-boots-estim}, the warping estimator (WarpE), c.f. \cite{Taylor:2007} and an estimator (IsotE) taylored to isotropic fields with the square exponential covariance function, c.f. \cite{Worsley:1996b} and \cite{Kiebel:1999}. The latter is currently implemented in the software package SPM12 used in the neuroimaging community.
We always use $1000$ Monte Carlo runs and consider two different scenarios for estimation of the LKCs. The {\it theoretical scenario}, assumes that the mean and variance are known and equal to zero and one, respectively, as required for the theory presented in Sections \ref{sec:Gaussestimator}. The {\it experimental scenario}, assumes that the mean and variance are unknown. Thus standardized residuals are used as input for the estimators as discussed in Sections \ref{sec:LKC-estim-asym} and \ref{sec:linear-models} instead of using the generated samples directly.

\paragraph{Isotropic Gaussian and non-Gaussian fields}

A common example of an isotropic Gaussian random field is
\begin{equation}\label{eq:isotropic-2D}
f(s) =  \frac{1}{\sqrt{\pi} \nu}\iint e^{-\frac{\|s-t\|^2}{2\nu^2}} W(dt)\,,\text{ for }s\in[1,L]^2\,,
\end{equation}
where $W$ is a Wiener field (white noise) on $\R^2$. Its covariance function is $C(s) = e^{-\beta \|t\|^2}$ with $\beta = 1/(4\nu^2)$. For an isotropic field on a compact domain $S \subset \R^2$ the LKCs can be derived  by $\L_0 = \L_0(S)$, $\L_1 = \lambda_2^{1/2} \L_1(S)$ and $\L_2 = \lambda_2 \L_2(S)$, where $\L_0(S)$, $\L_1(S)$, and $\L_2(S)$ are the LKCs of the domain $S$ with respect to the standard metric and $\lambda_2$ is the second spectral moment, equal to the variance of any directional derivative of $f$ \citep{RFG:2007}. For \eqref{eq:isotropic-2D} this yields $\L_0 = 1$, $\L_1 = 2\sqrt{2\beta}(L-1)$ and $\L_2 = 2\beta (L-1)^2$. Setting $\nu = 5$ and $L=50$ gives $\L_1 = 13.86$ and $\L_2 = 48.02$, which are used in this simulation.

In order to generate an iid sample $f_1(s),...,f_N(s)$ from \eqref{eq:isotropic-2D}, we approximate it by
\begin{equation}\label{eq:discr-isotropic-2D}
 f_n(s) =  \frac{\sum_{k,l=1}^L e^{-\frac{\|s-(k,l)\|^2}{2\nu^2}} W_{kl;n}}{\sum_{k,l=1}^L e^{-\frac{\|s-(k,l)\|^2}{\nu^2}}}\,,
\end{equation}
where $W_{kl;n}$, $k,l=1,...,L$, $n = 1,\ldots,N$ are iid random variables with mean 0 and variance 1. We consider a Gaussian case, where each of these variables is $N(0,1)$ and a non-Gaussian case where each variable is $\left(\chi^2_3-3\right)/\sqrt{6}$. Since LKCs depend only on the covariance structure, the LKCs for both choices of smoothed discrete random variables can be approximated by the LKCs of the field \eqref{eq:isotropic-2D}, if $\nu$ is large enough.

\begin{figure}[H]
\begin{center}
\begin{turn}{90}
\textbf{\hspace{0.5in}Theoretical}
\end{turn}
		\includegraphics[trim=50 0 40 30,clip,width=5.5in]{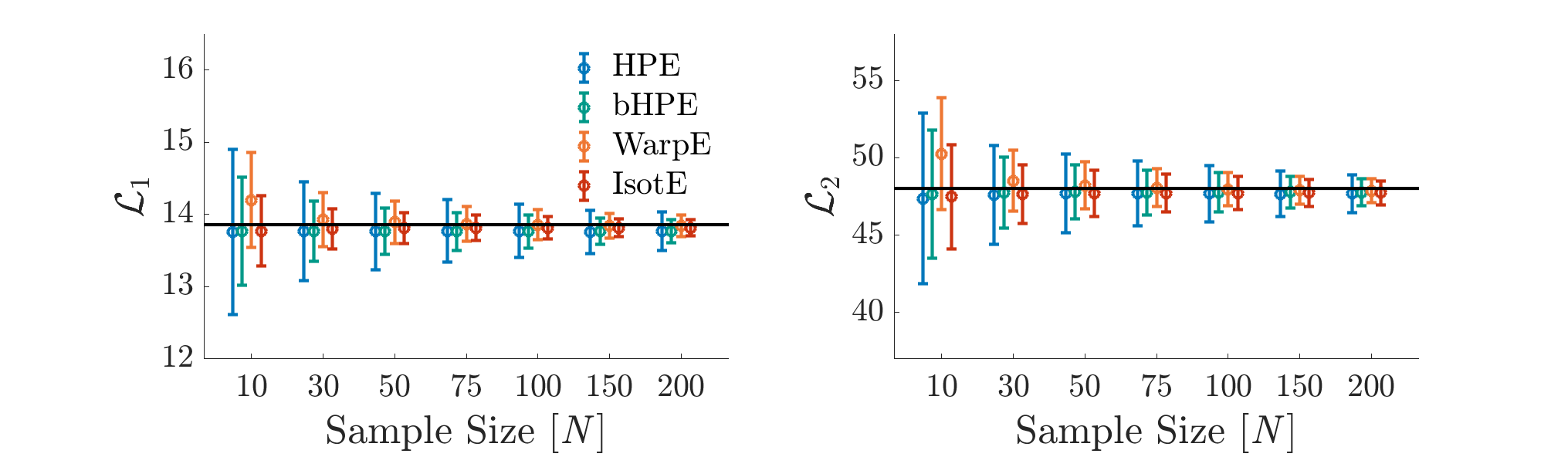}\\
		\vspace{0.1cm}
\begin{turn}{90}
\textbf{\hspace{0.5in}Experimental}
\end{turn}				
		\includegraphics[trim=50 0 40 30,clip,width=5.5in]{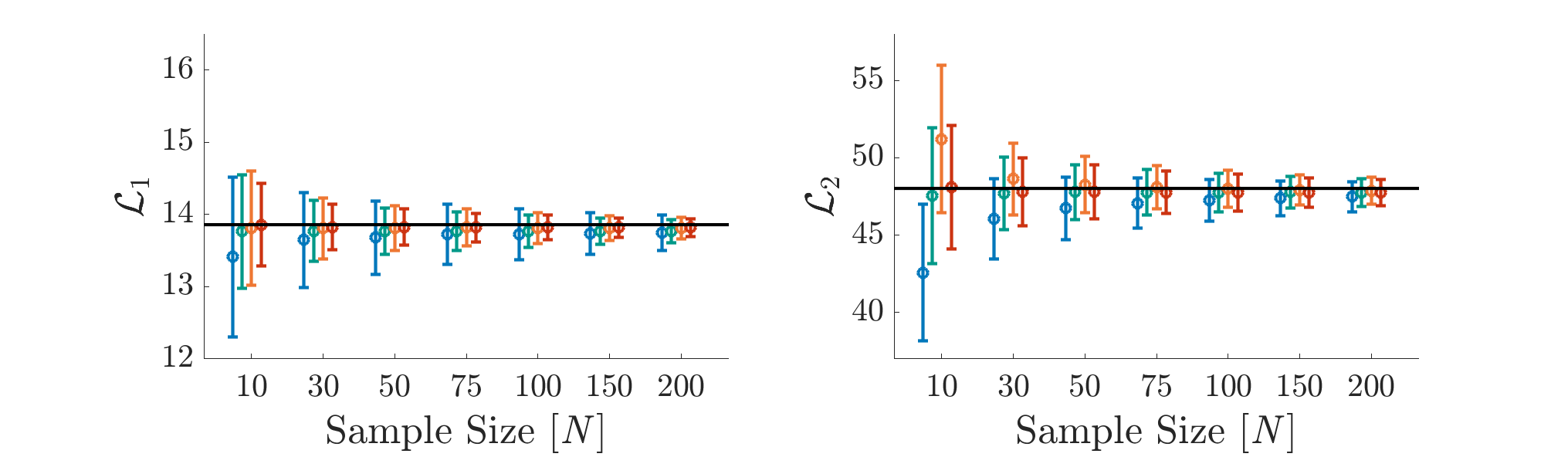}
\end{center}

\vspace*{-0.8cm}

	\caption{Isotropic Gaussian field ($\nu = 5$ and $L=50$): comparison of mean and standard deviation of different LKC estimators. Black lines represent the theoretical LKCs. 
	\label{fig:LKC-isotropic}}
\end{figure}

\paragraph{LKC estimation}
In the Gaussian case, Figure \ref{fig:LKC-isotropic} confirms the theoretical results that the HPE and the bHPE are consistent in the theoretical scenario. For every $N$, the HPE is almost unbiased in the theoretical scenario, as predicted by Theorem \ref{thm:UnbiasedAndVariance}\,(i), but the bHPE is almost unbiased in both the theoretical and experimental scenarios. Moreover, note that the bHPE has a similar variance as WarpE, which is considerably smaller than the variance of the HPE. The small downward bias seems to prevail for all estimators even for large $N$. To elucidate this, Table \ref{tab:Smoothness} shows the relative difference between the HPE and the theoretical LKCs of \eqref{eq:isotropic-2D}, which vanishes as $\nu$ increases. This indicates that the "bias" is due to the fact that the simulated field \eqref{eq:discr-isotropic-2D} has slightly different LKCs than the theoretical field \eqref{eq:isotropic-2D}. 

For non-Gaussian data, Figure \ref{fig:LKC-nongauss} shows that the bHPE is still unbiased. Only the HPE cannot handle this scenario, since the EC curves are not derived from a Gaussian field. WarpE and IsotE are mainly based on estimation of the covariance of the derivative of the random field and therefore are expected to still work for this particular non-Gaussian field.

\begin{figure}[H]
	\begin{center}
	\begin{turn}{90}
    \textbf{\hspace{0.5in}Theoretical}
    \end{turn}
	\includegraphics[trim=50 0 40 30,clip,width=5.5in]{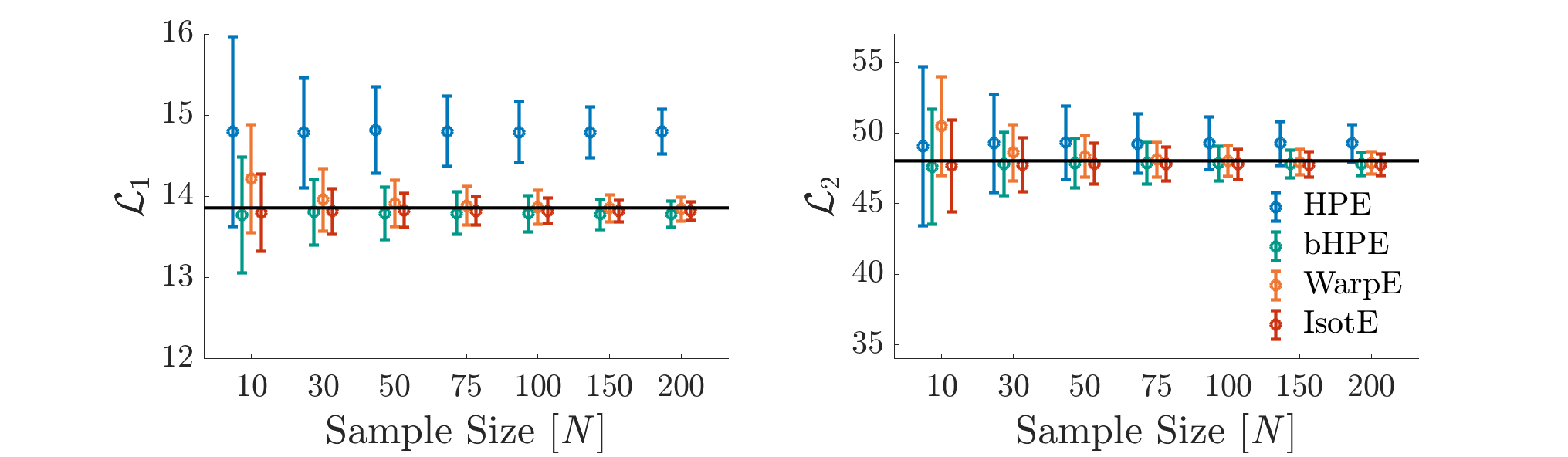}\\
		\vspace{0.1cm}
        \begin{turn}{90}
        \textbf{\hspace{0.5in}Experimental}
        \end{turn}		
		\includegraphics[trim=50 0 40 30,clip,width=5.5in]{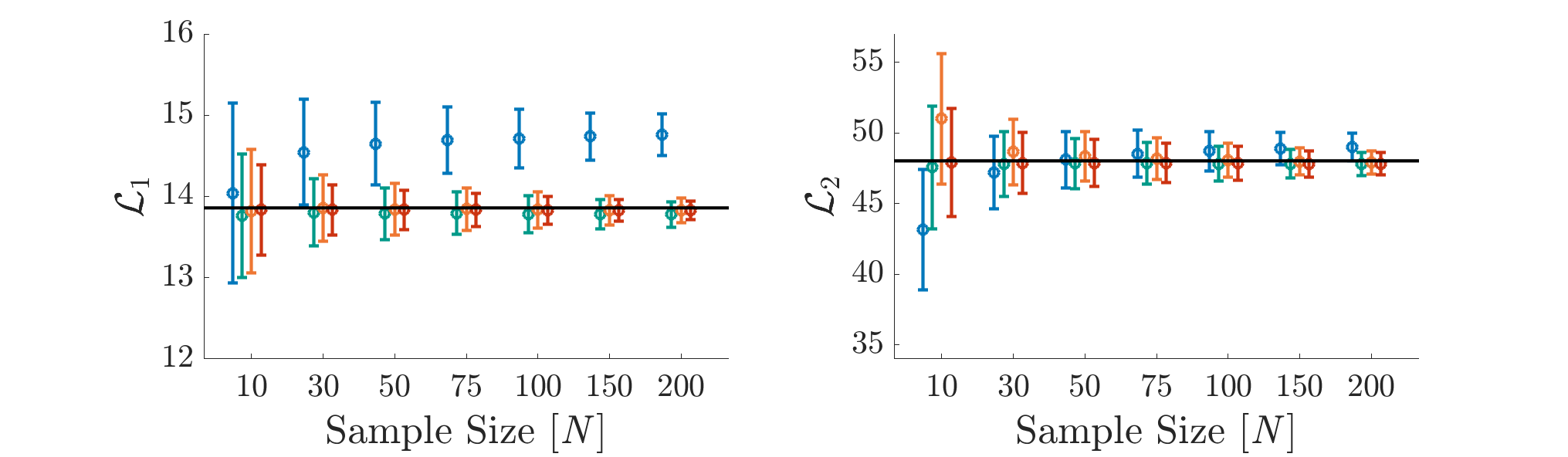}
	\end{center}
	\vspace*{-1.0cm}
	\caption{Isotropic non-Gaussian field ($\nu = 5$ and $L=50$): comparison of mean and standard deviation of different LKC estimators. Black lines represent the theoretical LKCs. 
	\label{fig:LKC-nongauss}}
\end{figure}

\begin{table}[t]\scriptsize
    \begin{center}
    \begin{tabular}{|c|ccccc|}
        \hline
        $\nu$ &       2 &       3  &      5   &     6    &   7\\\hline
    $\hat{\L}_1$, $N=10$ & -0.0195 & -0.0179 & -0.0075 & -0.0025 & 0.0011 \\
    $\hat{\L}_2$, $N=10$ & -0.0169 & -0.0090 & -0.0097 & -0.0069 & 0.0045 \\
    $\hat \L_1$, $N=75$ & -0.0240 & -0.0121 & -0.0075 & -0.0047 & 0.0012 \\
    $\hat \L_2$, $N=75$ & -0.0138 & -0.0076 & -0.0065 & -0.0065 & 0.0006\\\hline
    \end{tabular}
    \end{center}
    \vspace*{-0.5cm}
    \caption{ Comparison of relative bias $\E\big[\hat{\L}_d-\L_d\big] / \L_d$ for the HPE.}
    \label{tab:Smoothness}
\end{table}

\paragraph{HPE of the EEC} Moving on to the estimation of the EEC curve, the covariance function \eqref{eq:cov-chi-hat} of the HPE of the EEC curve $\widehat{\EEC}$ and $\bar{\chi}^{(100)}$ for a sample size of $N=100$ are shown in Figure \ref{fig:cov}. Taking the diagonal entries gives their variance functions, which are used to construct pointwise confidence bands.
Note that the HPE of the EEC has lower variance than the nonparametric sample average estimator, Figure \ref{fig:cov} (right panel).
\begin{figure}[H]
\begin{center}
  \includegraphics[trim=0 0 0 0,clip,height=1.7in]{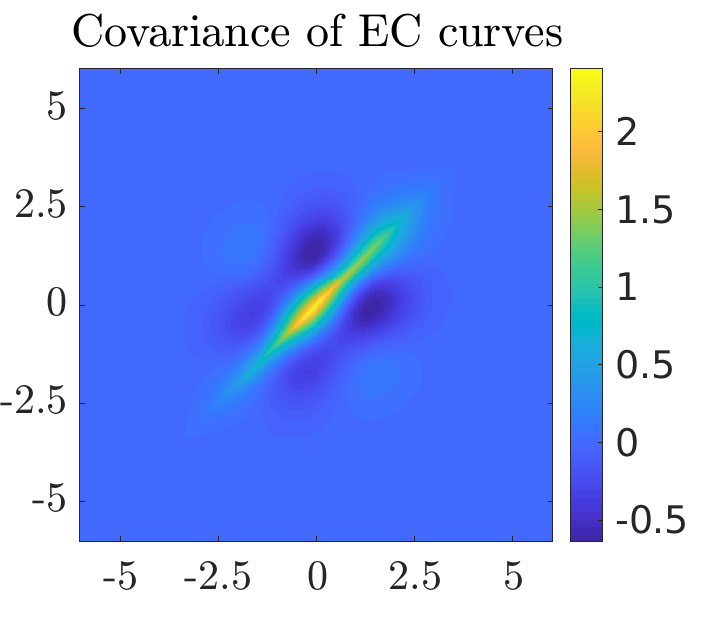}
  \includegraphics[trim=0 0 0 0,clip,height=1.7in]{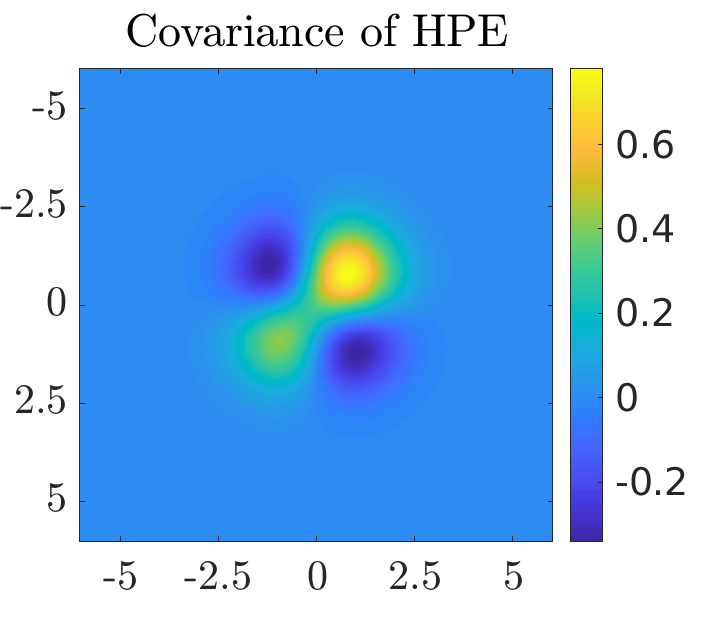}
  \includegraphics[trim=0 0 0 0,clip,height=1.7in]{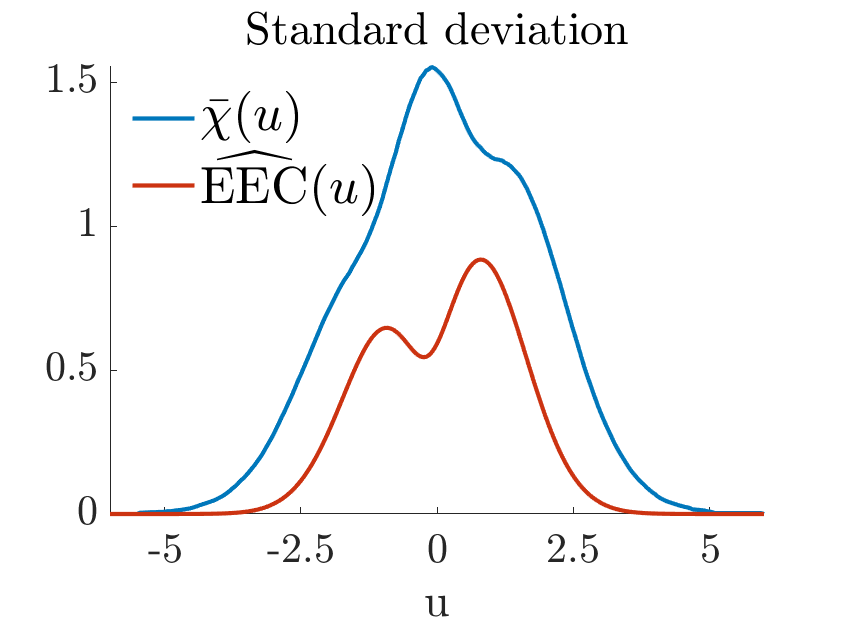}

\vspace*{-0.5cm}
\caption{Isotropic Gaussian field: Covariance functions of the sample average EC curve (left) and the HPE of the $\EEC$ (middle) from $1000$ samples. Right panel shows their standard deviation functions in red and blue, respectively.
\label{fig:cov}}
\end{center}
\end{figure}

\begin{figure}[H]
	\begin{center}
		\begin{tabular}{ccc}	
			 $N=10$ & $N=50$ & $N=100$ \\	
			\includegraphics[trim=0 0 0 0,clip,width=1.8in]{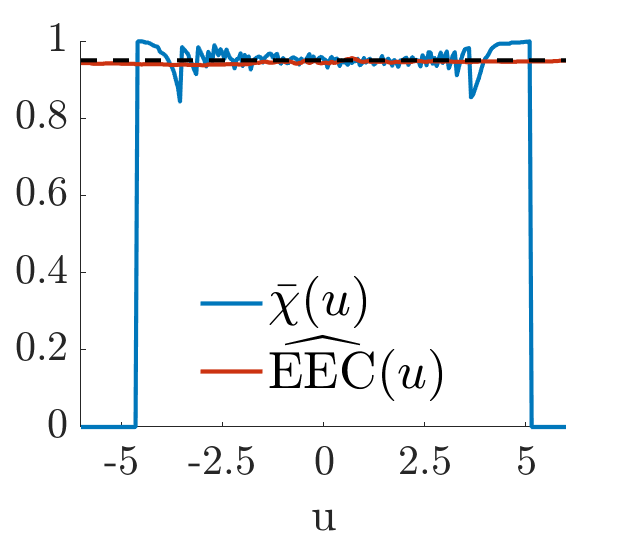} &
			\includegraphics[trim=0 0 0 0,clip,width=1.8in]{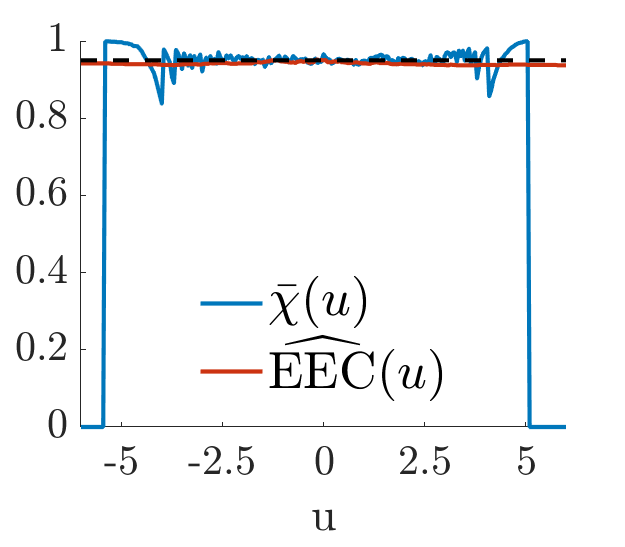} &
			\includegraphics[trim=0 0 0 0,clip,width=1.8in]{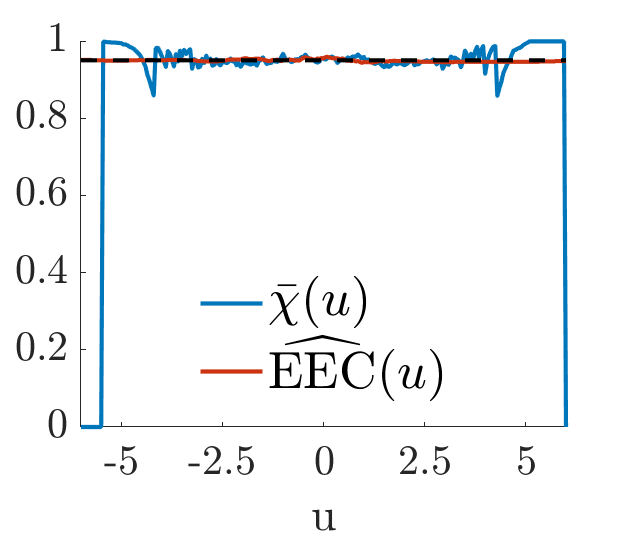} \\
			\includegraphics[trim=0 0 0 0,clip,width=1.8in]{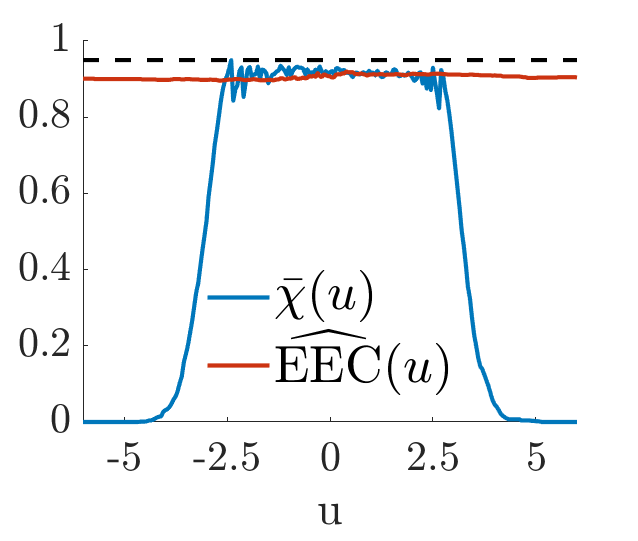} &
			\includegraphics[trim=0 0 0 0,clip,width=1.8in]{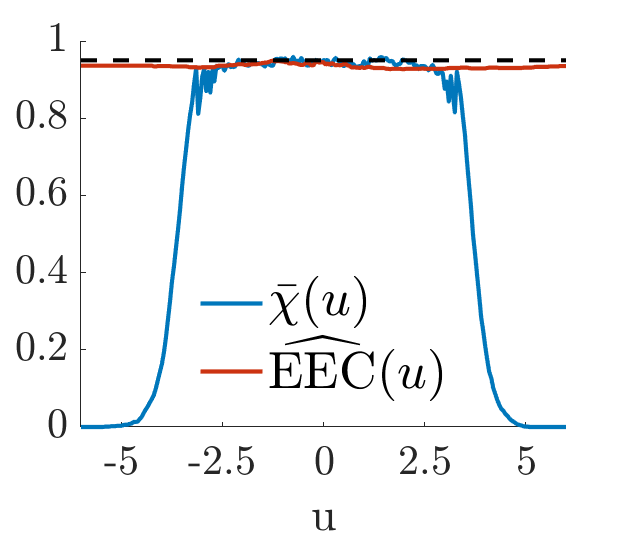} &
			\includegraphics[trim=0 0 0 0,clip,width=1.8in]{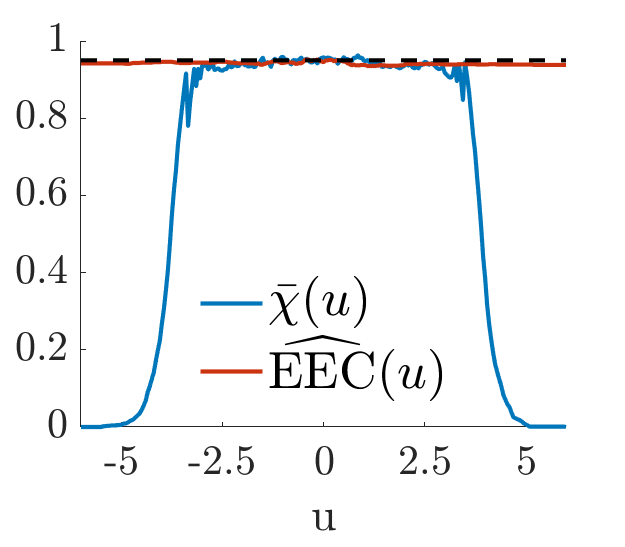} \\
		\end{tabular}
		
\vspace*{-0.3cm}
		\caption{Isotropic Gaussian field: pointwise coverage of the EEC curve under respectively ``true'' variance of the EC curves (top row) and estimated variance (bottom row) for different sample sizes. The dashed lines represent the target confidence level 95\%.  \label{fig:pointwise-coverage}}
	\end{center}
\end{figure}

\vspace*{-0.9cm}

Figure \ref{fig:pointwise-coverage} shows simulation results for the coverage of the true EEC curve when constructing pointwise $95\%$ confidence intervals via \eqref{eq:EEC-hat-conf} and \eqref{eq:ECavg-conf}. In the top row, the CIs use the ``true'' variance of the EC curves, estimated by Monte Carlo simulation for a large sample size of $10,000$, while in the bottom row, the CIs use the variance estimates corresponding to the given sample size. Obviously, the coverage function is smoother for the HPE (red) and it guarantees coverage for extreme values of $u$ especially when the variance is estimated from the data. The latter is not the case for the coverage of the CIs from the EC sample average, because large values of $u$ are seldomly observed.

\vspace{-0.5cm}
\paragraph{Non-stationary Gaussian field}
We simulate samples from the following field
\[
f(t, \gamma) =  \frac{1}{\pi^{1/4}\sqrt{\gamma}}\int e^{-\frac{(t-s)^2}{2\gamma^2}} W_n(ds),\quad \text{ for } (t,\gamma)\in [1,L]\times [\gamma_1, \gamma_2]
\]
where $t\in \R$, $\gamma\in (0, \infty)$ and the $W_n$'s are iid Wiener fields (white noise) on $\R$. This field is a unit-variance, smooth, non-stationary Gaussian random field, called {\it scale space field}. By \citet{Siegmund:1995}, its LKCs can be computed as
\begin{align*}
\L_1 &= \frac{L-1}{2\sqrt{2}}\left(\gamma_1^{-1}+\gamma_2^{-1}\right) + \frac{\log(\gamma_2/\gamma_1)}{\sqrt{2}}, ~~~
\L_2 = \frac{(L-1)}{2}\left(\gamma_1^{-1}-\gamma_2^{-1}\right).
\end{align*}

\vspace{-0.13cm}

\begin{figure}[H]
	\begin{center}
      \begin{turn}{90}
        \textbf{\hspace{0.5in}Theoretical}
        \end{turn}
	\includegraphics[trim=50 0 40 30,clip,width=5.5in]{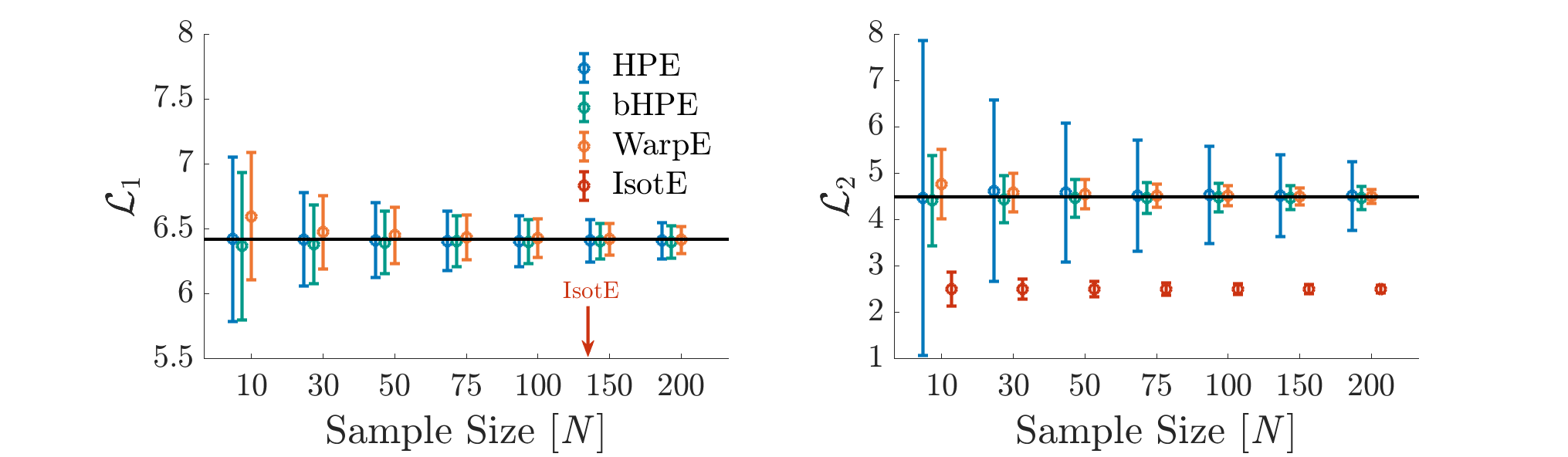}\\
        \begin{turn}{90}
        \textbf{\hspace{0.5in}Experimental}
        \end{turn}	
	\includegraphics[trim=50 0 40 30,clip,width=5.5in]{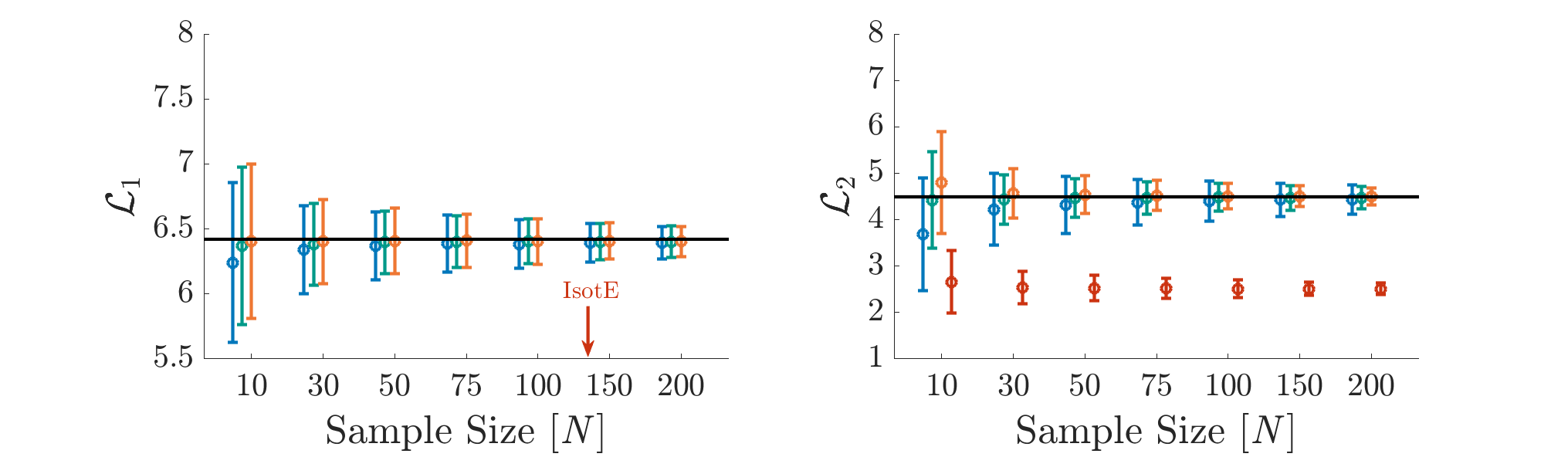}
	\caption{Scale space field: comparison of mean and standard deviation of different LKC estimators. Black lines represent the theoretical LKCs.  The values for the IsotE for $\L_1$ are not shown, since they are far off ($\hat\L_1 \approx 3.16$).
\label{fig:LKC-ScaleSpace}}
\end{center}
\end{figure}

In the simulations we again replaced the continuous field by a version derived from a discrete convolution using the parameters $L=50$, $\gamma_1=4$ and $\gamma_2=15$. This yields theoretical LKCs $\L_1 = 6.42$ and $\L_2 = 4.49$. Figures \ref{fig:LKC-ScaleSpace} through \ref{fig:ScaleSpace-pointwise-coverage} show the simulation results for the scale space following the same format as Figures \ref{fig:LKC-isotropic} through \ref{fig:pointwise-coverage}.
The results for this non-stationary field are similar to the isotropic case. In particular, the bHPE and the WarpE are best and perform similarly, except that the bHPE is unbiased. However, as expected the IsotE, which is the estimator currently used in neuroimaging, is unsuitable under strong nonstationarity.
\begin{figure}[H]
	\begin{center}
		\includegraphics[trim=0 0 0 0,clip,height=1.7in]{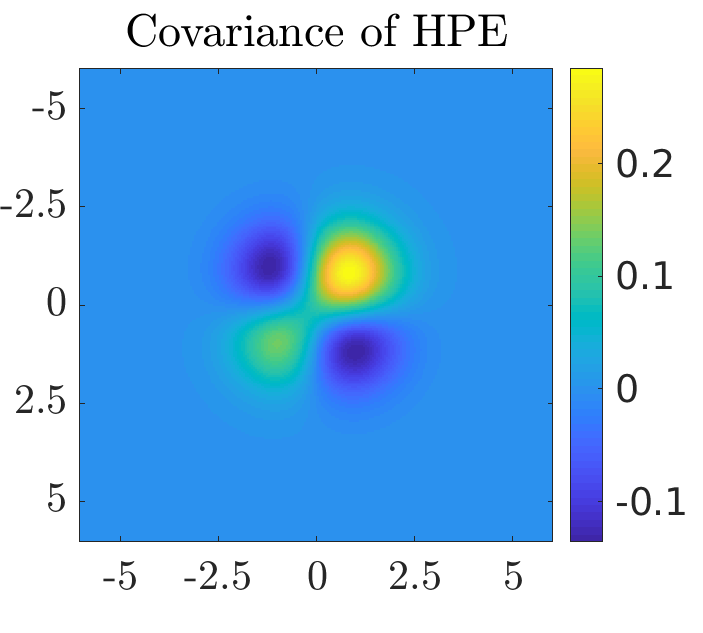}
		\includegraphics[trim=0 0 0 0,clip,height=1.7in]{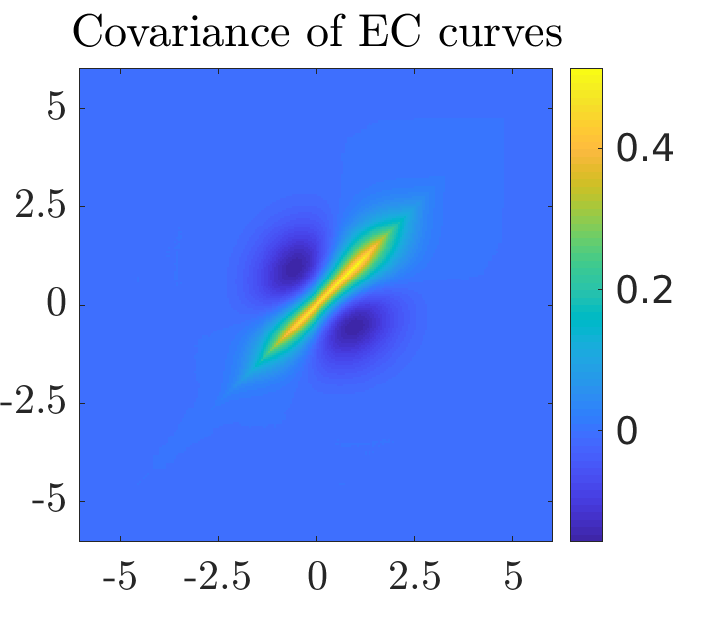}
		\includegraphics[trim=0 0 0 0,clip,height=1.7in]{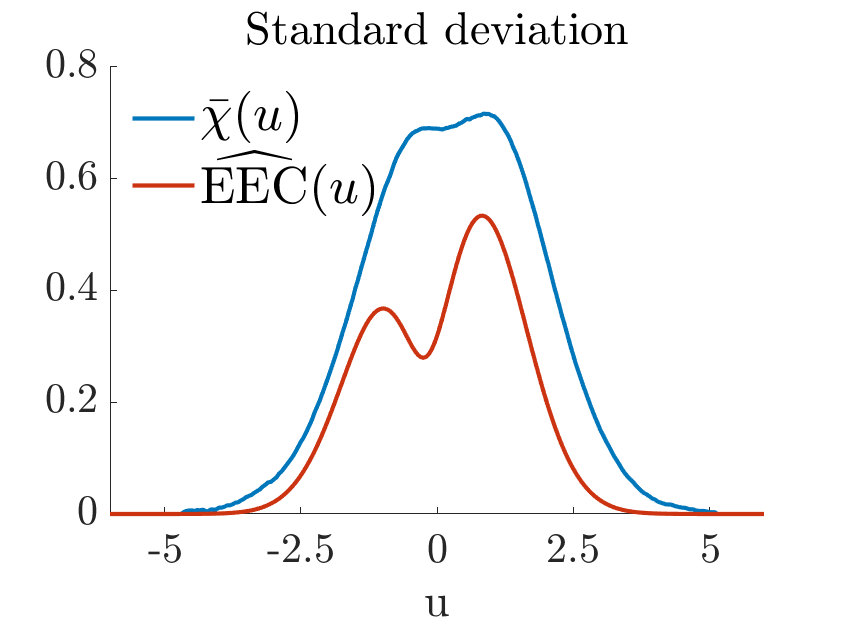}
			\vspace*{-0.5cm}
		\caption{Scale space: covariance functions of the sample average EC curve (left) and the HPE of the $\EEC$ (middle) from $1000$ samples. Right panel shows their standard deviation functions in red and blue, respectively.
			\label{fig:ScaleSpace_cov}}
	\end{center}
\end{figure}

\begin{figure}[H]
	\begin{center}
		\begin{tabular}{ccc}	
			$N=10$ & $N=50$ & $N=100$ \\	
			\includegraphics[trim=0 0 0 0,clip,width=1.8in]{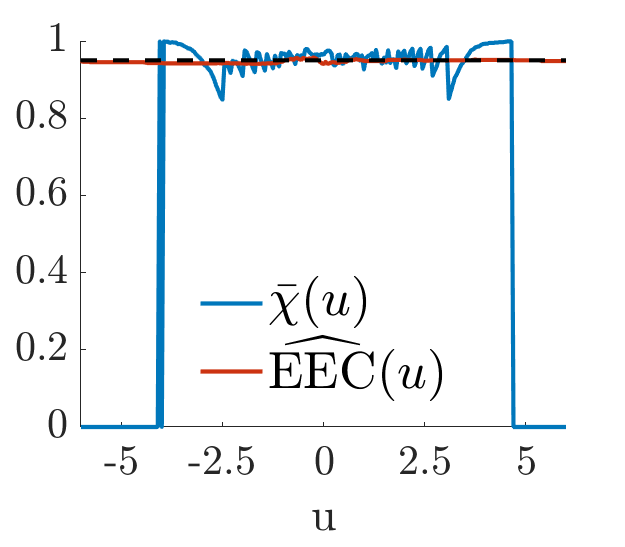} &
			\includegraphics[trim=0 0 0 0,clip,width=1.8in]{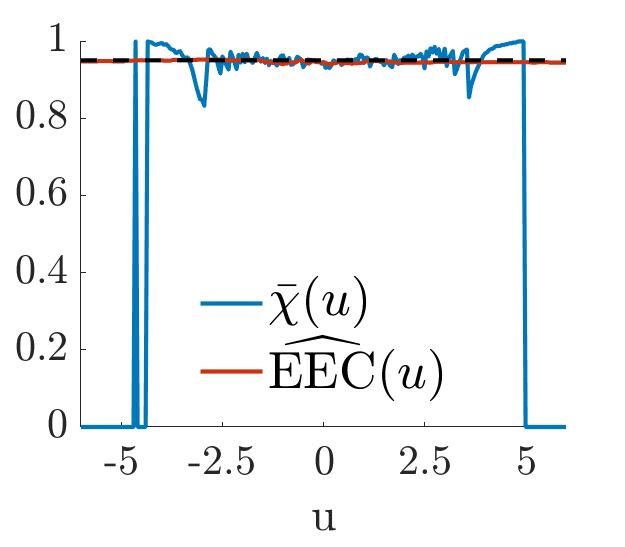} &
			\includegraphics[trim=0 0 0 0,clip,width=1.8in]{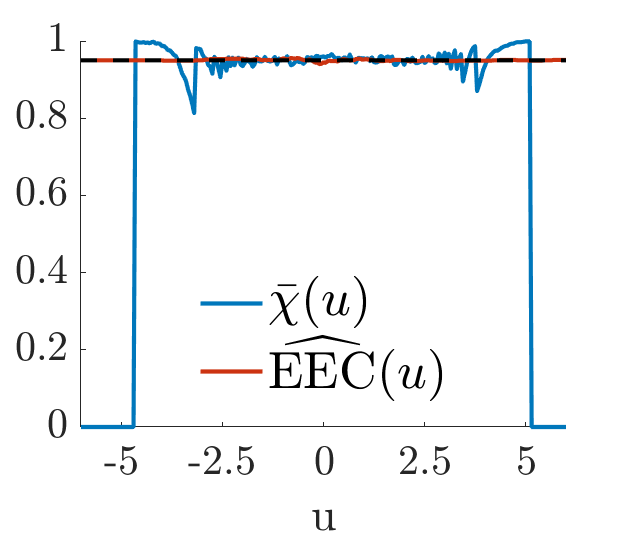} \\
			\includegraphics[trim=0 0 0 0,clip,width=1.8in]{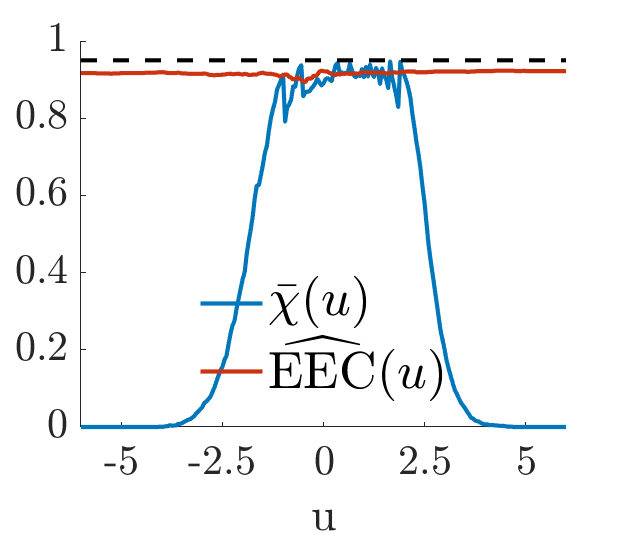} &
			\includegraphics[trim=0 0 0 0,clip,width=1.8in]{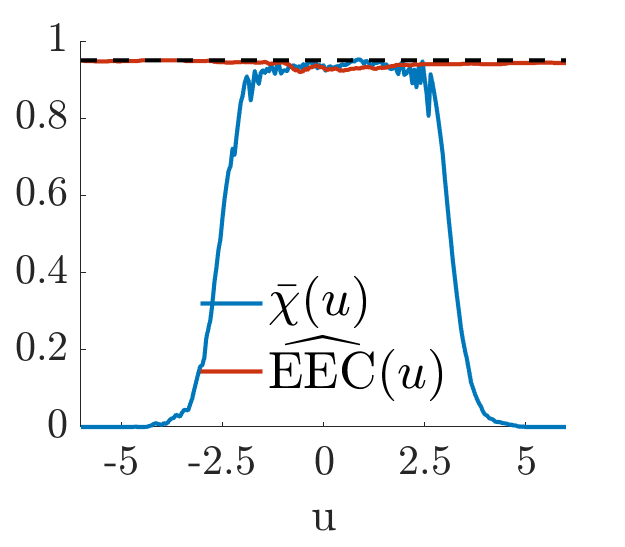} &
			\includegraphics[trim=0 0 0 0,clip,width=1.8in]{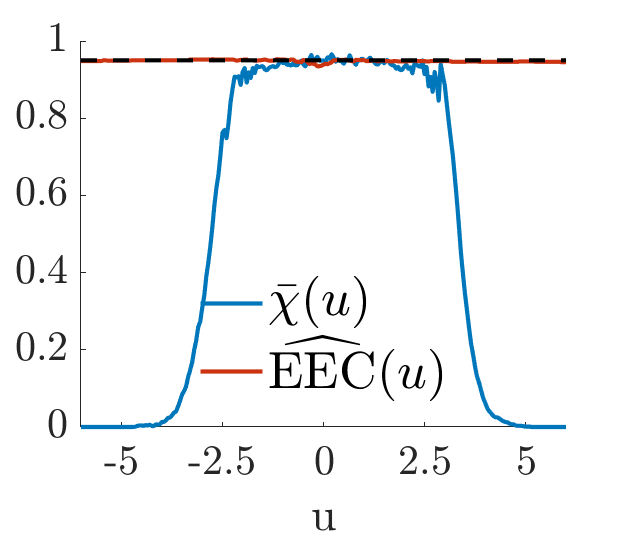} \\
		\end{tabular}				
		\caption{Scale space:  pointwise coverage of the EEC curve under respectively ``true'' variance of the EC curves (top row) and estimated variance (bottom row) for different sample sizes. The dashed lines represent the target confidence level 95\%.  \label{fig:ScaleSpace-pointwise-coverage}}
	\end{center}
\end{figure}

\section{Cosmic Microwave Background Radiation}\label{sec:data}

In cosmology, the EEC plays an important role in summarizing the topological structure of the universe \citep{GDM86,HSKT02,PEW16,Pranav:2019}. In particular, the cosmic microwave background (CMB) radiation field gives a glimpse into the early structure of the universe, a short time after the big bang \citep{PlanckXXIII}. The goal is to compare the EC of the observed CMB map to the EEC of simulated Gaussian fields with the same spatial autocorrelation function, helping assess how close the theoretical 
\begin{figure}[H]
	\begin{center}
			\includegraphics[trim=0 0 0 0,clip,width=3in]{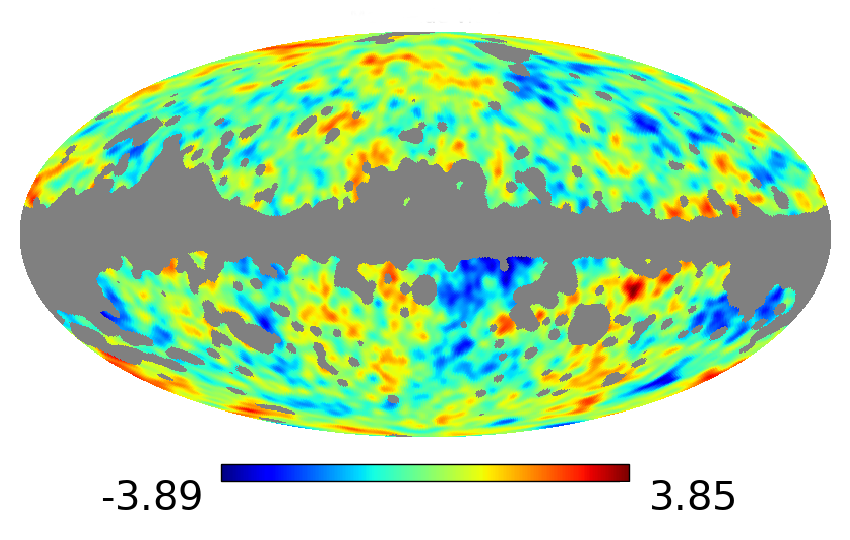}\\
			\includegraphics[trim=40 0 40 0,clip,width=2.6in]{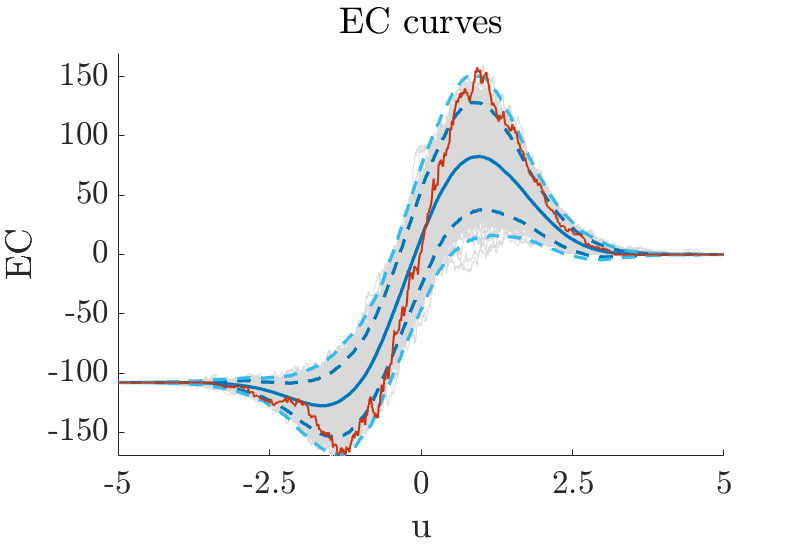} 
			\includegraphics[trim=40 0 40 0,clip,width=2.6in]{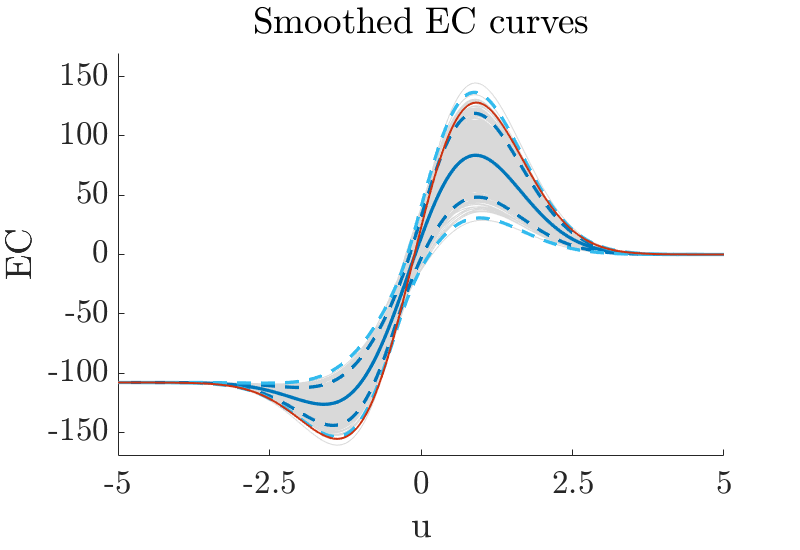}		
	\end{center}
	
	\vspace*{-1.1cm}
	\caption{ \label{fig:CMB-EEC} CMB example: (top) cleaned and smoothed observed CMB field. (left/right) EC curves, raw/ smoothed, simulated CMB fields (gray) and nonparametric EEC estimate/HPE, with pointwise 2/3-$\sigma$-predicition bands (blue/light blue). The EC curve corresponding to the Planck data is shown in red.}
\end{figure}
\noindent model is to the observed data. For this comparison, we use $N=1000$ iid instances of Full Focal Plane 8 (FFP8) simulations, publicly released by the Planck team \citep{Planck:2016-ffp8}. These simulations are designed to replicate the observed CMB sky and are based on an isotropic Gaussian random field prescription. To mimic the observation field, in our computer experiments, the CMB maps were contaminated by noise from various sources and then cleaned using the same process as the actual observed data from Planck \citep{Planck:2016-NILC}, including smoothing with an isotropic Gaussian kernel with a bandwidth of 180 arcmin. Additionally, because of contamination by the Milky Way and other large foreground light sources, portions of the sky were masked using the most conservative UT78 mask released by the Planck team \citep{PlanckXXIII}. Figure \ref{fig:CMB-EEC} shows the observed CMB field rendered over a HEALPix spherical grid with 2048 nodes \citep{HEALPix:2005}.

The simulated fields and the observed field were thresholded at a sequence of thresholds ranging from -5 to 5 in steps of 0.01 and the EC of the resulting excursion sets were computed. Figure \ref{fig:CMB-EEC} shows the observed EC curves and their smoothed versions, together with the nonparametric estimate and the HPE of the EEC. We decided to use the HPE instead of the bHPE here, since the simulated data is in fact Gaussian and we can provide estimates of the standard errors (the bHPE is used in the fMRI data analysis). The estimated LKCs from the 1000 FFP8 simulations were $\hat{\L}_1 = 426.8 \pm 54.2$ and $\hat{\L}_2 = 1528.9 \pm 346.4$, yielding respective standard errors of $1.7$ and $11.0$ for the mean LKCs. In contrast, the estimated LKCs from the actual observed field were $\hat{\L}_1^o = 480.3$ and $\hat{\L}_2^o = 2545.0$.

These results have two uses in astronomy. First, the EEC could be used to detect celestial objects against the CMB background \citep{peakDet-Sphere:2019}. In this case, we obtain $u_{\FWER} = 4.279 \pm 0.002$ and $u_{\CER} = 3.456 \pm 0.002$.
Second, in Figure \ref{fig:CMB-EEC}, the observed CMB EC curve (red) is at the edge of the distribution of simulated curves. In fact, the observed $\hat{\L}_2^o$ is 2.93 standard deviations away from the corresponding estimated mean value $\hat{\L}_2$ from the simulations. This may be evidence that the observed CMB field does not match the physical model that generated the simulations and improvements to the physical model may be necessary \citep{PEW16,Pranav:2019}.

\section{Discussion}

The main advantage of the HPE/bHPE is the ability to handle non-stationary and non-Gaussian random fields in a conceptually and computationally simple way. Moreover, the fields are allowed to be defined over non-trivial domains, like subdomains of the sphere or the highly curved cortical surface of the brain, provided that there is an algorithm to compute the EC of excursion sets. In fact, general and efficient algorithms to compute EC curves exist and are a recent topic in computer science \citep{Heiss:2017}.

While the WarpE performs well in our simulations and is a more direct, geometric approach, the bHPE is simpler and equally effective. In particular, its reliance on EC curves rather than approximations of numerical derivatives makes it computationally favorable, since the WarpE can be prone to numerical inaccuracies. Noteworthy is that the bHPE has a similar variance to the WarpE in our simulations, and in contrast to the latter, seems to be unbiased even in the experimental scenario. Observing similar variances is not surprising, since both estimators compute the LKCs of the Gaussian field having the empirical correlation of the data. Hence uncertainty in these estimates is determined by the variability in the empirical correlation. It is not evident why the uncertainty in the HPE, which instead is determined by the variability of the EC curves, is higher.

A further advantage of the HPE/bHPE is that we were able to study theoretical properties as unbiasedness, finite variance, consistency and derive CLTs and confidence bands for the EEC, which was only partially feasible for the WarpE, see \citet{Taylor:2007}.

\section*{Acknowledgments}
A.S., D.C. and F.T. were partially supported by NIH grant R01EB026859. F.T. thanks the WIAS Berlin, where parts of the research for this article was performed, for providing a guest researcher status and their hospitality. P.P. acknowledges the support of ERC advanced grant 740021- Advances in Research on THeories of the dark UniverSe (ARTHUS).

\bigskip
\begin{center}
{\large\bf SUPPLEMENTARY MATERIAL}
\end{center}

\begin{description}
\item[Proofs:] All proofs of the theorems in the main manuscript.
\item[Appendix:] Additional material on the connection of the HPE to LKC regression, efficient computation of EC curves, LKC estimation for general linear models and a data analysis application to fMRI data.
\item[Matlab Toolbox for Estimation:] \url{https://github.com/ftelschow/HermiteProjector}  Matlab Toolbox implementing the proposed estimators and allowing to reproduce the simulations and data analysis.
\end{description}

\bibliographystyle{abbrvnat}

\bibliography{references.bib}
\newpage

\appendix
  
\section{Appendix}
\subsection{HPE and LKC regression}\label{appendix:lin-reg-view}
\cite{EC-regression2017} suggested to perform a linear regression of the average field on the EC-densities based on the GKF \eqref{eq:EEC} to estimate the LKCs. Suppose $\bar{y}(u) = \bar{\chi}^{(N)}(u) -  \L_0 \Phi^+(u)$ is the pinned average of the observed EC curves of an iid sample $f_1,...,f_N\sim f$. Then

\begin{equation}\label{eq:adlerRegression}
\bar y(u) := \sum_{d=1}^D \L_{d} \rho_d(u) + \varepsilon(u)\,,
\end{equation}
where $\varepsilon(u)$ is some mean zero error field with covariance structure $(u,v)\mapsto\Cov\big[ \chi_f(u),\chi_f(v) \big]$.

Following that formulation, $\bar{y}(u)$ is observed at a discrete set of $L$ levels $u_1,\ldots,u_L$ to get the response vector $\bar{\mathbf y} = (\bar{y}(u_1),\ldots,\bar{y}(u_L))^{\tt T}$. Similarly, the columns of the design matrix $X$ are the EC densities sampled at the same levels so that
$X_{dl} = \rho_{d}(u_l)$. The linear regression estimator of the vector of LKCs $\L$ is
\begin{equation}
\label{eq:OLS}
\hat{\L}_{\rm LR} = (X^{\tt T} X)^{-1} X^{\tt T} \bar{\mathbf y}.
\end{equation}

The HPE can be interpreted as a continuous version of the linear regression estimator. First, notice that the linear regression estimate $\hat{\L}_{\rm LR}$, by definition, is the vector $\beta = (\beta_1,\ldots,\beta_D)^{\tt T}$ that minimizes the sum of squares
$\sum_{l=1}^L \left[y_l - \vec{x}_l \beta\right]^2$,
where $\vec{x}_l$ is the $l$-th row of the matrix $X$. Respecting the functional form of the EC process, a continuous version of the above sum of squares is the integral
$\int_{-\infty}^\infty \left[\bar{y}(u) - \vec{x}(u) \beta\right]^2 w(u)\,du$
where $\vec{x}(u) = (\rho_1(u),\ldots,\rho_D(u))$ and $w(u)$ is a suitable weight function. Differentiating with respect to $\beta$ and setting to zero yields that the minimizer is
\begin{equation}
\label{eq:beta-hat}
\hat{\beta} = \left[\int_{-\infty}^\infty \vec{x}(u)^{\tt T} \vec{x}(u) w(u) \,du\right]^{-1}
\left[\int_{-\infty}^\infty \vec{x}(u)^{\tt T} \bar{y}(u)\, w(u) \,du\right].
\end{equation}
This solution can be greatly simplified by judiciously choosing $w(u) = e^{u^2/2}$. In this case, the $(d,d')$ entries of the $D\times D$ matrix on the left of \eqref{eq:beta-hat} are precisely given by \eqref{eq:orthogonality-rho}.
This matrix is diagonal and immediately invertible yielding that each of the entries of \eqref{eq:beta-hat} is the same as the HPE \eqref{eq:L-hat-equiv}. In other words, the HPE \eqref{eq:L-hat-equiv} can be seen as a LKC regression estimator in the sense of \cite{EC-regression2017} with weight function $e^{u^2/2}$, so that high and low levels $u$ are weighted more heavily than levels $u$ near zero. This may be seen as an advantage in practice because, if the estimated EEC curve is to be used for inference, it is more important to have a better fit at high and low thresholds.

\subsection{Efficient computation of EC curve using critical values}\label{appendix:ECcomputation}
Computationally, most challenging for implementing the HPE and the bHPE is the computation of the EC curves, which, if not judiciously implemented can be computationally inefficient. We present a simple algorithm computing the piecewise constant EC curve from Theorem \ref{thm:CritValRepresentation} exactly for a discretely observed field $f$.

We use two key facts. First, due to Morse theory it is well known, that the topology of the level sets of $f$ change if and only if the level $u$ is a critical value of $f$. Secondly, the Euler characteritic is local meaning that the change in the Euler characteristic of the set $f^{-1}\big((u-\epsilon, u]\big)$ compared to the set $f^{-1}\big((u-\epsilon, u]\big)\backslash f^{-1}\big(\{u\}\big)$ for a critical value $u$ and small enough $\epsilon$ is equal to the change between the ECs $\chi_f(u-\epsilon)$ and $\chi_f(u)$ of the corresponding level sets.

Computationally, these facts are used as follows. For simplicity of presentation, we assume that the field is observed on a $2$ dimensional discrete square grid with $L$ points in each direction. Arbitrary dimensions are completely analoguosly treated and more complex domains $S$ than a square can always be incorporated by fitting it into a rectangular domain $S\subset D$ and define $f(s)=-\infty$, if $s\in D\backslash S$. Assume that $F\in\R^{ (L+2)\times (L+2)}$ is the array with entries $F_{ij}=f(i,j)$ for $i,j\notin \{1, L+2 \}$ and $f(i,j)=-\infty$ else. The padding of $-\infty$ simplifies the implementation. Further, we assume that we have an algorithm computing the Euler characteristic of a binary $3\times 3$ array, let us denote it as a function $E_3:\{0,1\}^{3\times 3}\rightarrow \mathbb{Z}$. This can be fastly implemented and there are different choices depending on the chosen definition of the EC for a discrete grid. In our implementation we assume that the observation points are the vertices and use $4$-connectivity, i.e., only vertical and horizontal vertices are connected. Equivalently, $8$-connectivity is possible where we also connect diagonal vertices. For $F_{ij}$ we define its local neighbourhood by the array $N^{(ij)}\in\R^{ 3\times 3}$ consisting of  $F_{ij}$ and its $8$ neighbours.

The following construction describes the implementation.
\begin{enumerate}
	\item Using a $for$-loop compute for all $(i,j)$ such that $i,j\notin \{1, L+2 \}$ the local topology change at $(i,j)$, i.e., $\Delta_{i,j} = E_3\big( N^{(ij)} \geq F_{ij} \big)- E_3\big( N^{(ij)} > F_{ij} \big)$ and save the tupel $(F_{ij},\Delta_{i,j})$. Denote the set of these tupels by $\mathcal{T}$, which contains all tupels of critical values $F_{ij}$ and their topology change.
	\item The EC curve can be recovered from $\mathcal{T}$ as follows. Sort all tupels in $\mathcal{T}$ ascending in the critical values. We call the sorted values $\big(u_{(0)}, \Delta_{(0)}\big),..., \big(u_{(M)}, \Delta_{(M)}\big)$, where $M$ is the number of critical values. The step function representation \eqref{eq:EC-crit-val-repr} of the EC curve is then obtained by
	\begin{equation}
		\chi_f(u) = \L_0\mathds{1}_{(-\infty,u_{(0)}]}(u) + \sum_{m=1}^M a_m \mathds{1}_{(u_{(m-1)},u_{(m)}]}(u)\,,
	\end{equation}
	where $a_m =  \L_0 + \sum_{k=0}^{m-1} \Delta_{(k)}$.
\end{enumerate}

\subsection{EEC estimation for pointwise linear models}
\label{sec:linear-models}

For fMRI data in particular \citep{Worsley:1996,Worsley:2004,Nichols:2012}, but also in other settings \citep{Sommerfeld:2018}, it is customary to set up a linear regression model relating the observed random fields at each location $s$. Following what is usually called the general linear model approach, the $N\times 1$ vector $\Y(s)$ of observed intensities at each location $s$ is modeled as
\begin{equation}\label{eq:linear-model}
\Y(s) = \X \bbeta(s) +  \bepsilon(s),
\end{equation}
where the $P\times 1$ vector $\bbeta(s)$ contains $P$ regression coefficients, $\X$ is a constant $N\times P$ design matrix independent of $s$, and $\bepsilon(s)$ is a $N\times 1$ random vector whose entries are assumed to be iid with zero mean and some variance $\sigma^2(s)$. In a task-related fMRI experiment, $\Y(s)$ represents the observed fMRI images at $N$ time points and $\X$ contains the onset of the various tasks conditions in addition to other covariates, so that the function $\bbeta(s)$ represents the extent to which the observed signal is related to the task at each location $s$. The goal of the inference is to estimate a specific contrast $\eta(s) = \mathbf{c}'\bbeta$, for a fixed vector $\mathbf{c}$, typically comparing task conditions, and idenify for example areas in the brain, where the contrast is different from zero. The detection of such regions of activation is often done by thresholding a standardized estimate of $\eta(s)$, so that the FWER across the brain is controlled.

Specifically, let $\hat{\bbeta}(s) = (\X'\X)^{-1}\X'\Y(s)$ be the usual pointwise least-squares estimate of the coefficient vector $\bbeta(s)$ at each location $s$, with covariance matrix $\Cov\big[\hat{\bbeta}(s)\big] = \sigma^2(s) (\X'\X)^{-1}$. The estimate of $\eta(s)$ is $\hat{\eta}(s) = \mathbf{c}'\hat{\bbeta}(s)$, with variance $\Var\left[\hat{\eta}(s)\right] = \sigma^2(s) \mathbf{c}'(\X'\X)^{-1}\mathbf{c}$. Inference at each location $s$, for example to test the null hypothesis $H_{0}: \ \eta(s)=0$ at each location $s$, may be based on the standardized z-score field
\begin{equation}\label{eq:z-score}
z(s) = \frac{\hat{\eta}(s)}{\sqrt{\Var\left[\hat{\eta}(s)\right]}}.
\end{equation}
Let $\rho(s,s') = \E[\epsilon_n(s),\epsilon_n(s')]$, $s,s'\in S$, be the correlation function of the noise field. A simple calculation shows that the random field $z(s)$ has constant variance 1 and covariance function equal to $\rho(s,s')$. Under the complete null hypothesis that $H_{0}: \ \eta(s)=0$ for all $s\in S$, the random field $z(s)$ also has zero mean and the detection threshold $u_\alpha$ can be obtained using the EEC heuristic as explained in the previous section.

If the error field $\epsilon$ is smooth and the design matrix is appropriately bounded, then the z-score field \eqref{eq:z-score} satisfies a fCLT, with a smooth Gaussian limiting field satisfying the GKF \eqref{eq:EEC} and the LKCs being the same as those of the error field $\bepsilon$  \citep{Sommerfeld:2018}.

Ideally, estimation of these LKCs using the HPE or bHPE would be based on $N$ iid realizations of the error field. Because these are not available, we construct standardized residuals as discussed in Section \ref{sec:bHPEdef}. We define normalized residuals $\tilde{\e}(s) = \e(s)/\sigma(s)$, where $\e(s) = \Y(s) - \X \hat{\bbeta}(s)$.
These fields are smooth and have mean zero and covariance function $\E[\tilde{\e}(s)\tilde{\e}'(t)] = [\I - \X(\X'\X)^{-1}\X'] \rho(s,t)$. While the fields given by the components are not independent, classical results in linear regression (e.g. \citep{Eicker:1963}) show that asymptotic properties still hold, if the sample size $N$ is replaced by the number of degrees of freedom $N-P$. For large $N$, the variance estimate $\hat{\sigma}^2(s) = \|\e(s)\|^2/(N-P)$ should be close to its true value. Therefore, and since the normalized residuals are not observable, we base our estimates on the standardized residual fields, i.e. the entries of $\mathbf{R}(s) = \e(s)/\|\e(s)\|$, which satisfy \textbf{(R1)} and under appropriate assumptions also \textbf{(R2)}. Note again, that even if the error field $\bepsilon$ is Gaussian, the standardized residuals $\hat{\e}$ are not. For this reason, the bHPE from Section \ref{sec:LKC-estim-asym} performs better than the HPE estimator from Section \ref{sec:observed-LKC}, as we will demonstrate in our simulations.

\subsection{fMRI data analysis: Nonstationary 3D random field}
\label{sec:fmri}

In fMRI analysis, the use of EEC for controlling FWER has been mostly restricted to calculations based on stationarity of the error field. In this example, we show how the methods proposed in this paper can be used to obtain the significance threshold without assuming stationarity and taking into account the variability in the EEC estimation.

The fMRI data \citep{Moran:2012}, obtained from the public repository OpenfMRI ({\tt openfmri.org}), is the same analyzed in \cite{peakDetRF:2017}, allowing us to compare the results. The dataset involves a ``false belief task" experiment, where subjects read short stories concerning a person's false belief about reality or stories about false realities not involving people. The goal is to find brain regions that show a contrast in neural activity between these two situations, and can thus be attributed to processing  other people's false beliefs about reality.

As in \cite{peakDetRF:2017}, we focus here on the data from subject \# 49. The data $\Y(s)$ consists of a sequence of $N = 179$ fMRI images of size $71\times 72 \times 36$ voxels, after motion correction, spatial registration and removal of the first row (missing data). The design matrix $\X$ contains 4 columns encoding the presentation of the stimuli as 0-1 step functions, in addition to a column for the intercept. The vector $\mathbf{c}$ encodes the contrast of interest between the two types of stimuli. Following the analysis described in Section \ref{sec:linear-models}, we computed the regression residual fields $\e_i(s)$, which we smoothed with a Gaussian kernel with standard deviation of 1.6 voxels and the corresponding z-score field $z(s)$ from equation \eqref{eq:z-score}.

\begin{figure}[H]
	\begin{center}
		\begin{tabular}{cc}
			\includegraphics[trim=0 2 0 0,clip,width=3in]{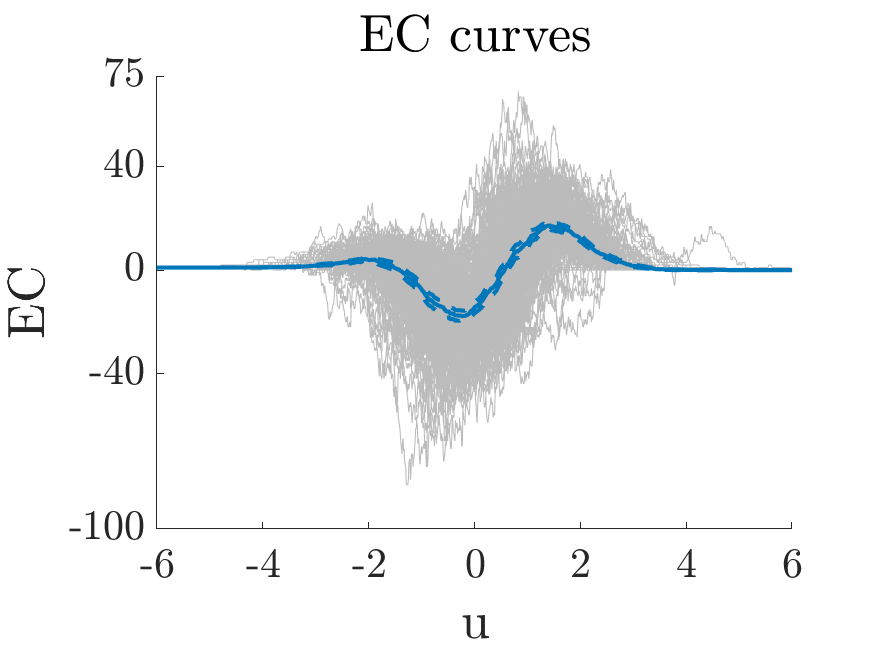} 
			&\includegraphics[trim=0 2 0 0,clip,width=3in]{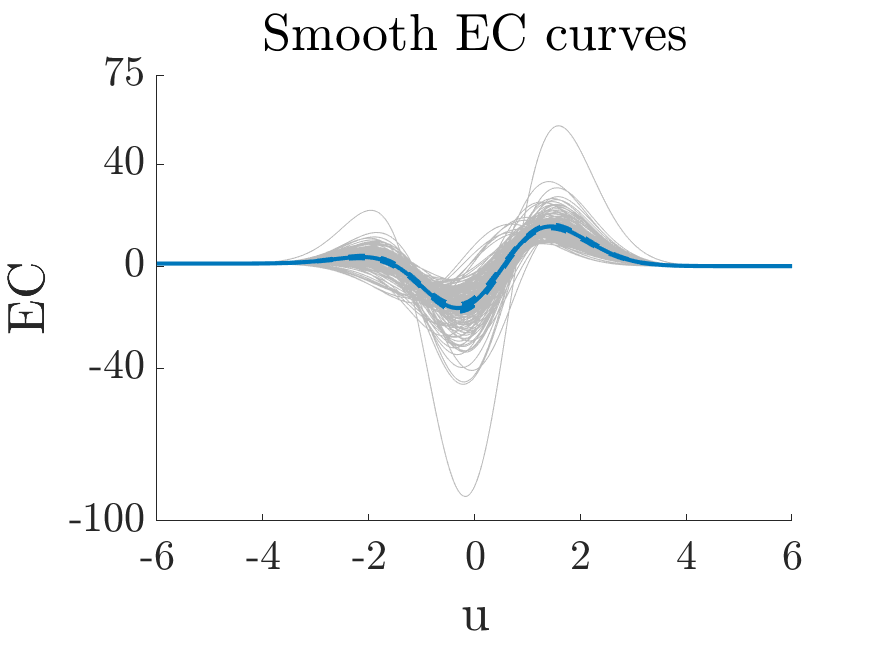}		
		\end{tabular}
	\end{center}
	\caption{ \label{fig:fMRI-EEC} fMRI example: Observed EC curves and estimates with pointwise confidence bands. (left) nonparametric EEC estimate. (right) HPE of EEC.}
\end{figure}

To compute EC curves, the spatial domain $S$ was defined as a brain mask composed of all voxels with raw fMRI activity greater than $500$. Figure \ref{fig:fMRI-EEC} shows the observed and smoothed EC curves for the $N=179$ re-normalized residual fields restricted to the brain mask $S$, together with the nonparametric and smooth EEC estimates. As expected, the smooth EEC estimate has tighter confidence bands. To perform statistical inference, we proceed as in Section \ref{sec:threshold}. The results are summarized in Table \ref{table:DataResults}. Here $u_\FWER$ and $u_\CER$ are the detection thresholds corresponding to $\alpha=0.05$ and $\alpha=1$, respectively. Note also that we only provide standard errors for the LKCs estimates of the HPE, which were obtained as the square root of the diagonal entries of $\hat{\bSigma}/N$, as described in Section \ref{sec:repeated-obs}. The LKC estimates for bHPE and HPE are fairly close suggesting that deviation from Gaussianity is not strong. However, the estimates from the software package SPM12 \cite{}, which implements IsotE, are substantially different due to non-isotropicity and non-stationarity. Interestingly, in this experiment, this does not affect the detection threshold substantially.

\begin{table}[H]\scriptsize
    \begin{center}
    \begin{tabular}{|c|ccc|cc|cc|}
        \hline
                  &       $\hat{\L}_1$ &       $\hat{\L}_2$  & $\hat{\L}_3$   &  $u_{\FWER}$    &   $\chi(z>u_{\FWER})$ &  $u_{\CER}$    &   $\chi(z>u_{\CER})$\\\hline
            HPE   & $13.5 \pm 1.2$    & $261.3 \pm 7.0$  & $650.7 \pm 27.0$ & $4.21\pm0.01$ & $5$ & $3.28\pm0.01$ & $22$ \\
            bHPE  & $13.2$ & $266.7$  & $670.1$ & $4.21$ & $5$ & $3.29$ & $21$\\
            SPM12 & $35.8$ & $315.3$  & $669.0$ & $4.23$ & $5$ & $3.31$ & $20$\\\hline
    \end{tabular}
    \end{center}
    \caption{ Comparison of LKC estimates, corresponding thresholds obtained from the GKF together with EC of excursion sets from the Moran data for HPE, bHPE and SPM12 (IsotE).}
    \label{table:DataResults}
\end{table}

\begin{figure}[H]
\begin{center}
	(a)\includegraphics[width=0.40\textwidth]{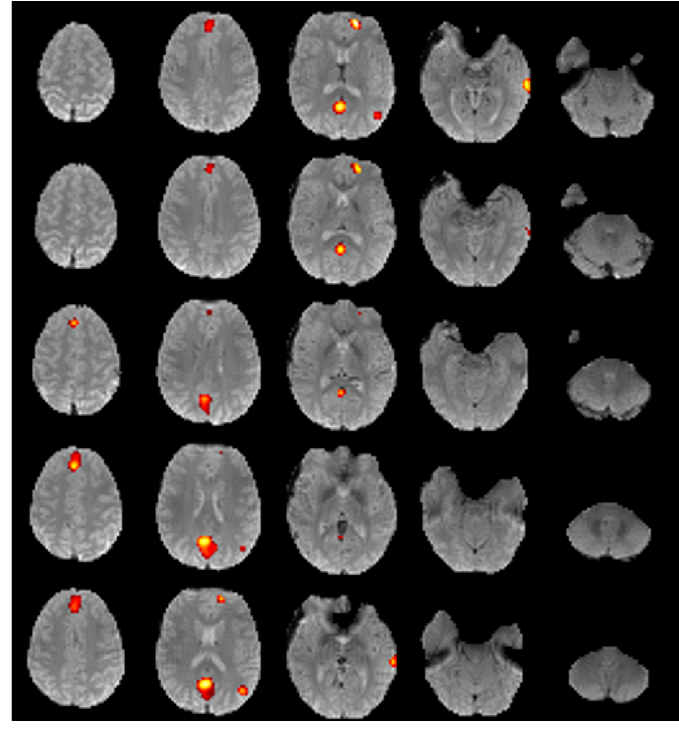}\hspace{0.5cm}
	(b)\includegraphics[width=0.40\textwidth]{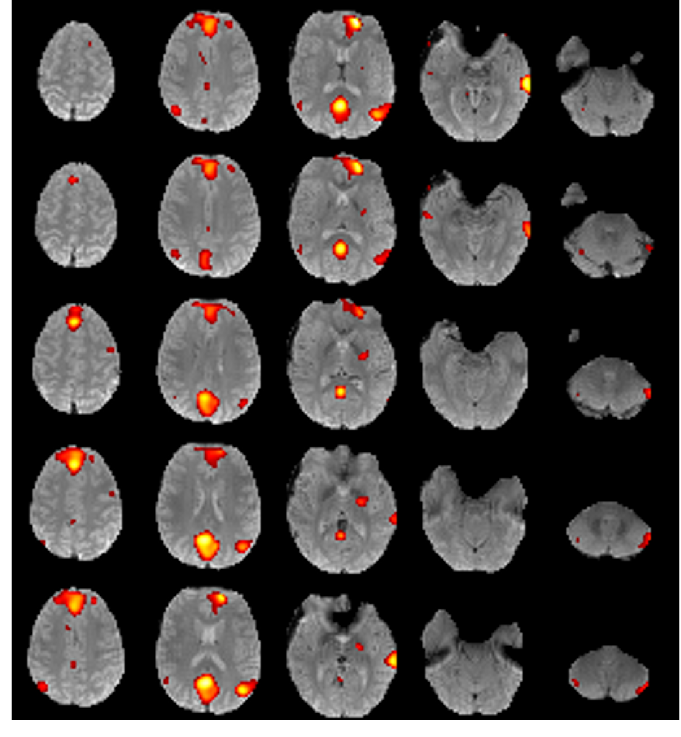}
\end{center}
	\caption{ \label{fig:fMRI-activation} fMRI activation at a significance level of FWER=0.05 (left) and CER=1 (right), using the residuals within the brain volume to estimate the EEC. Montage shows the brain volume as transverse slices from the top of the brain (top left panel) to the bottom of the brain (bottom right panel). Colored regions indicate the smoothed Wald statistic field above the threshold. Results are superimposed on an anatomical brain image (gray) for reference. }
\end{figure}

Figure \ref{fig:fMRI-activation} shows the activation maps for the thresholds from the bHPE. The FWER map is conservative; the excursion set has EC $\chi(u_{\FWER})=5$ with no apparent holes, equivalent to 5 connected components, which show as 5 hotspots in the activation map. In comparison, the CER map is less conservative; the excursion set has EC $\chi(u_{\CER})=22$, again with no apparent holes, equivalent to 22 connected components. The interpretation is that we expect only one connected component to be false.

\section{Proofs}\label{appendix:proofs}
      \subsection{Proof of Theorem \ref{thm:CritValRepresentation}}
            Note that $u_0=\min_{s\in S} f(s)$ and $u_M=\max_{s\in S} f(s)$, then we obtain
            \begin{align}\label{eq:hatL-v1}
            \begin{aligned}
                \tfrac{(d-1)!}{(2\pi)^{d/2}}\hat\L_d &=\int_{-\infty}^\infty \!H_{d-1}(u)\left( \chi_f(u)-\L_0\Phi^+(u) \right)\,du\\ &= \int_{u_0}^{u_M}\! H_{d-1}(u)\chi_f(u)\,du - \L_0\int_{u_0}^{\infty}\! H_{d-1}(u)\Phi^+(u)\,du  \\
                                        &~ ~ ~+ \L_0\int_{-\infty}^{u_0}\! H_{d-1}(u)\left(1-\Phi^+(u)\right)\,du\,.
            \end{aligned}
            \end{align}
            Here we used that $\chi_f(u)=0$, if $u>u_M$ and $\chi_f(u)=\L_0$, if $u<u_0$.

            Using integration by parts and $H_d'(u)=d\cdot H_{d-1}(u)$ yields for the later two integrals
            \begin{align*}
                \int_{u_0}^{\infty}\! H_{d-1}(u)\Phi^+(u)\,du &= -\Phi^+(u_0)\tfrac{H_{d}(u_0)}{d} - \int_{u_0}^{\infty}\! \tfrac{H_{d}(u)}{d} {\Phi^+} '(u)\,du \\
                \int_{-\infty}^{u_0}\! H_{d-1}(u)\left(1-\Phi^+(u)\right)\,du &= \big(1-\Phi^+(u_0)\big) \tfrac{H_{d}(u_0)}{d} + \int_{-\infty}^{u_0}\! \tfrac{H_{d}(u)}{d}{\Phi^+}'(u)\,du
            \end{align*}
            Using the definition of Hermite polynomials $H_d(x) = (-1)^de^{\tfrac{x^2}{2}}\frac{d^d}{dx^d}e^{-\tfrac{x^2}{2}}$ yields
            \begin{align*}
            \int_{-\infty}^{\infty}\! \tfrac{H_{d}(u)}{d}\tfrac{e^{-u^2/2}}{\sqrt{2\pi}}\,du = \tfrac{(-1)^d}{d\sqrt{2\pi}}\int_{-\infty}^{\infty}\! \tfrac{d^d}{dx^d}e^{-\tfrac{x^2}{2}}\,du =  \tfrac{(-1)^d}{d\sqrt{2\pi}} \left[\tfrac{d^{d-1}}{dx^{d-1}}e^{-\tfrac{u^2}{2}}\right]_{-\infty}^\infty = 0\,.
            \end{align*}
            Therefore, equation \eqref{eq:hatL-v1} simplifies to
            \begin{align*}
            \tfrac{(d-1)!}{(2\pi)^{d/2}}\hat\L_{d} &= \int_{u_0}^{u_M}\! H_{d-1}(u)\chi_f(u)\,du + \L_0\tfrac{H_{d}(u_0)}{d} - \L_0\int_{-\infty}^{\infty}\! \tfrac{H_{d}(u)}{d}\tfrac{e^{-u^2/2}}{\sqrt{2\pi}}\,du\\
            &= \int_{u_0}^{u_M}\! H_{d-1}(u)\chi_f(u)\,du + \L_0\tfrac{H_{d}(u_0)}{d}
            \end{align*}
            and thus using the representation \eqref{eq:EC-crit-val-repr} and $H_d'(x) = d\cdot H_{d-1}(x)$ we finally compute:
            \begin{align*}
            \tfrac{d!}{(2\pi)^{d/2}}\hat\L_{d} &= \L_0H_{d}(u_0) + \sum_{m=1}^{M} a_m \big( H_d(u_{m}) - H_d(u_{m-1}) \big)\\
                                                &= \L_0H_{d}(u_0) + \sum_{m=1}^{M} a_m H_d(u_{m}) -  \sum_{m=0}^{M-1} a_{m+1} H_d(u_{m})\\
                                                &= \sum_{m=0}^{M} (a_m-a_{m+1}) H_d(u_{m})\,,
            \end{align*}
            where we defined $a_0=\L_0$ and $a_{M+1}=0$.
            \qed

\subsection{Proof of Theorem \ref{thm:UnbiasedAndVariance}}
Part 1):

 By equation \eqref{eq:L-recover} we have to justify interchanging the integral and the expectation in
    \begin{align*}
        \tfrac{(d-1)!}{(2\pi)^{d/2}}\E\left[ \hat\L_{d} \right] &= \E\left[ \int_{-\infty}^\infty \!H_{d-1}(u)\left( \chi_f(u)-\L_0\Phi^+(u) \right)\,du \right] \\
         &= \int_{-\infty}^\infty \!H_{d-1}(u)\left( \EEC(u)-\L_0\Phi^+(u) \right)\,du = \tfrac{(d-1)!}{(2\pi)^{d/2}} \L_{d}\,.
    \end{align*}
 Therefore, we consider the following splitting of the integral
 \begin{align*}
 \tfrac{(d-1)!}{(2\pi)^{d/2}}\hat\L_{d}&= \int_{u_0}^{u_M}\! H_{d-1}(u)\chi_f(u)\,du 
                        - \L_0\int_{u_0}^{\infty}\! H_{d-1}(u)\Phi^+(u)\,du\\
                        &\quad+ \L_0\int_{-\infty}^{u_0}\! H_{d-1}(u)\left(1-\Phi^+(u)\right)\,du\nonumber\\
                       &= I + \L_0\cdot II + \L_0\cdot III\,,
\end{align*}
where as above $u_0=\min_{s\in S} f(s)$. For each summand we show seperately that we can interchange the expectation and the integral, if we write it using characteristic functions.

We begin with integral I. Here it is sufficient to prove that
\begin{equation}
    \int_{-\infty}^{\infty}\!\E\left[  \mathds{1}_{(u_0,u_M]}(u) \vert H_{d-1}(u)\chi_f(u) \vert  \right]\,du < \infty\,.
\end{equation}
Again we use the stepfunction representation \eqref{eq:EC-crit-val-repr} and note that $\vert a_m \vert \leq M_0 = M+\L_0$ for all $m=1,...,M$, since at each critical point the EC can only change by $\pm 1$. Thus, using $(G4a)$, the triangle inequality, H\"older's inequality and Borel-TIS inequality, we obtain for any $\epsilon>0$
\begin{align*}
     &\int_{-\infty}^{\infty}\!\E\left[  \mathds{1}_{(u_0,u_M]}(u) \vert H_{d-1}(u)\chi_f(u) \vert  \right]\,du\\
        &\quad\quad\quad\quad\leq \int_{-\infty}^{\infty}\!\E\left[  \sum_{m=1}^{M} \vert a_m \vert \mathds{1}_{(u_{m-1},u_{m}]}(u)\right] \vert H_{d-1}(u) \vert \,du \\
        &\quad\quad\quad\quad<    \int_{-\infty}^{\infty}\!\E\left[  M_0 \mathds{1}_{(u_{0},u_{M}]}(u)\right] \vert H_{d-1}(u) \vert \,du \\
        &\quad\quad\quad\quad\leq \int_{-\infty}^{\infty}\!\E\left[  M_0^{1+\epsilon}\right]^{\frac{1}{1+\epsilon}} \E\left[  \mathds{1}_{(u_{0},u_{M}]}(u)\right]^{\frac{1}{1+1/\epsilon}} \vert H_{d-1}(u) \vert \,du \\
        &\quad\quad\quad\quad\leq C'\int_{-\infty}^{\infty}\! P\left[ \min_{s\in S} f(s) < u \leq \max_{s\in S} f(s)\right]^{\frac{1}{1+1/\epsilon}} \vert H_{d-1}(u) \vert \,du \\
        &\quad\quad\quad\quad\leq C'\int_{-\infty}^{\E[u_0]}\! e^{\frac{-(u-\E[u_0])^2}{2+2/\epsilon}} \vert H_{d-1}(u) \vert \,du + C'\int_{\E[u_0]}^{\E[u_M]}\! \vert H_{d-1}(u) \vert\,du + \\
        &\quad\quad\quad\quad~~~~ C'\int_{\E[u_M]}^{\infty}\! e^{\frac{-(u-\E[u_M])^2}{2+2/\epsilon}} \vert H_{d-1}(u) \vert \,du < \infty\,.
\end{align*}
Thus, for I we are allowed to interchange integral and expectation. Note that we used
\begin{align*}
    P\left[ \min_{s\in S}f(s)<u \right] &\leq e^{-\frac{(-u+\E[u_0])^2}{2\sigma^2_T}}\,,~~~ \text{ for } u < \E\left[ \min_{s\in S}f(s)\right] \\
    P\left[ \max_{s\in S}f(s)>u \right] &\leq e^{-\frac{(u-\E[u_M])^2}{2}}\,,~~~ \text{ for } u > \E\left[ \max_{s\in S}f(s)\right]\,,
\end{align*}
which are easily derived from Borel-TIS inequality (e.g., \citet[Thm. 2.1.1]{RFG:2007}). Note that $\sigma^2_T=1$ in our case.

Next consider II and III. Note that similarly as above using Borel-TIS and the exponential bounds for $\Phi^+(u)$ and $1-\Phi^+(u)$, we compute
\begin{align*}
   &\int_{-\infty}^{\infty}\! \E\left[ \mathds{1}_{(u_0,\infty)}(u)\vert H_{d-1}(u)\vert \Phi^+(u) \right]\,du\\
      &\quad\quad\quad\quad\leq \int_{-\infty}^{\infty}\!  \vert H_{d-1}(u)\vert\Phi^+(u)P\left[ \min_{s\in S} f(s) < u \right]\,du \\
      &\quad\quad\quad\quad\leq \int_{-\infty}^{\E[u_0]}\! e^{\frac{-(u-\E[u_0])^2}{2}} \vert H_{d-1}(u) \vert \,du + \int_{\E[u_0]}^{\infty}\!  \vert H_{d-1}(u)\vert\Phi^+(u)\,du<\infty\,.\\
\end{align*}
and
\begin{align*}
   &\int_{-\infty}^{\infty}\! \E\left[ \mathds{1}_{(-\infty, u_0)}(u)\vert H_{d-1}(u)\vert \big(1-\Phi^+(u)\big) \right]\,du\\
      &\quad\quad\quad\quad\leq \int_{-\infty}^{\infty}\!  \vert H_{d-1}(u)\vert\big(1-\Phi^+(u)\big)P\left[ \min_{s\in S} f(s) > u \right]\,du \\
      &\quad\quad\quad\quad\leq \int_{-\infty}^{\E[u_M]}\! \big(1-\Phi^+(u)\big) \vert H_{d-1}(u) \vert \,du \\
      &\quad\quad\quad\quad~ ~ ~+ \int_{\E[u_M]}^{\infty}\!  \vert H_{d-1}(u)\vert P\left[ \max_{s\in S} f(s) > u \right]\,du < \infty\,.
\end{align*}

Part 2):

Note that by part 1) we have
\begin{equation*}
    \sigma_{dd'} = \Cov\left[\hat{\L}_{d}, \hat{\L}_{d'}\right] = \E\left[ \hat{\L}_{d}\hat{\L}_{d'} \right] - \L_{d}\L_{d'}\,.
\end{equation*}
Hence we will only show that the expectation on the r.h.s. is finite. With slight but obvious change in notation to part 1) we can write
\begin{align*}
    \tfrac{(d-1)!^2}{(2\pi)^{d}}\E\left[ \hat{\L}_{d}\hat{\L}_{d'} \right] &= \E\left[ (I_d + \L_0\cdot II_d + \L_0\cdot III_d)(I_{d'} + \L_0\cdot II_{d'} + \L_0\cdot III_{d'}) \right]
\end{align*}
Our strategy is again to bound each summand seperately and even more show that we could interchange expectation and integration, since this implies immediately the finiteness. The arguments are very similar to the unbiasedness proof. Therefore we shorten the computations considerably. Moreover, we use the abbreviation $P(u,u')=\vert H_{d-1}(u) H_{{d'}-1}(u')\vert$.

Case $\E\big[ I_dI_{d'} \big]$:
\begin{align*}
     \iint_{\R^2}\!&P(u,u')\E\left[  \mathds{1}_{(u_{0},u_{M}]}(u)\mathds{1}_{(u_{0},u_{M}]}(u') \vert\chi(u)\chi_f(u') \vert  \right]\,dudu'\\
        &<\iint_{\R^2}\!P(u,u')\E\left[  M_0^2 \mathds{1}_{(u_{0},u_{M}]}(u)\mathds{1}_{(u_{0},u_{M}]}(u')\right] \,dudu' \\
        &\leq\iint_{\R^2}\!P(u,u')\E\left[  M_0^{2+\epsilon} \right]^{\frac{2}{2+\epsilon}} E\left[  \mathds{1}_{(u_{0},u_{M}]}(u)\mathds{1}_{(u_{0},u_{M}]}(u')\right]^{\frac{1}{1+2/\epsilon}} \,dudu' \\
        &\leq\iint_{\R^2}\!P(u,u')\E\left[  M_0^{2+\epsilon} \right]^{\frac{2}{2+\epsilon}} E\left[  \mathds{1}_{(u_{0},u_{M}]}(u)\right]^{\frac{1}{2+4/\epsilon}} \E\left[ \mathds{1}_{(u_{0},u_{M}]}(u')\right]^{\frac{1}{2+4/\epsilon}} \,dudu' \\
        &\leq C\int_{-\infty}^{\infty}\vert H_{d-1}(u)\vert P\left[ \min_{s\in S} f(s) < u \leq \max_{s\in S} f(s)\right]^{\frac{1}{2+4/\epsilon}} du\\
        &  ~ ~~ ~ ~ ~ ~ ~ ~ ~ ~\cdot\int_{-\infty}^{\infty}\!\vert H_{k-1}(u')\vert P\left[ \min_{s\in S} f(s) < u' \leq \max_{s\in S} f(s)\right]^{\frac{1}{2+4/\epsilon}} \,du' <\infty
\end{align*}

Case $\E\big[ II_dII_{d'} \big]$:
\begin{align*}
   \iint_{\R^2}\! &P(u,u') \Phi^+(u)\Phi^+(u') \E\left[ \mathds{1}_{(u_0,\infty)}(u)\mathds{1}_{(u_0,\infty)}(u') \right]\,dudu' \\
      &\leq \iint_{\R^2}\!  P(u,u') \Phi^+(u)\Phi^+(u') P\left[ \min_{s\in S} f(s) < u \right]^{\frac{1}{2}}P\left[ \min_{s\in S} f(s) < u \right]^{\frac{1}{2}}\,dudu' \\
      &\leq \int_{\R}\!  \vert H_{d-1}(u)\vert \Phi^+(u)P\left[ \min_{s\in S} f(s) < u \right]^{\frac{1}{2}}\,du \cdot\int_{\R}\vert H_{k-1}(u')\vert\Phi^+(u') P\left[ \min_{s\in S} f(s) < u \right]^{\frac{1}{2}}\,du' \\
      &< \infty\,.
\end{align*}

Case $\E\big[ III_dIII_{d'} \big]$: Reduces again basically, by the same arguments as in the previous case to the integrals III.

Case $\E\big[ I_dII_{d'} \big]$:
\begin{align*}
    \iint_{\R^2}\!&P(u,u')\Phi^+(u')\E\left[  \mathds{1}_{(u_{0},u_{M}]}(u)\mathds{1}_{[u_{0},\infty)}(u') \vert\chi(u) \vert  \right]\,dudu'\\
    &\leq\iint_{\R^2}\!P(u,u')\Phi^+(u') \E\left[  M_0^{2+\epsilon}\right]^{\frac{2}{2+\epsilon}} \E\left[  \mathds{1}_{(u_{0},u_{M}]}(u)\right]^{\frac{1}{2+4/\epsilon}}\E\left[  \mathds{1}_{[u_{0},\infty)}(u)\right]^{\frac{1}{2+4/\epsilon}} \!dudu'\\
    &\leq C\int_{-\infty}^{\infty}\vert H_{d-1}(u)\vert P\left[ \min_{s\in S} f(s) < u \leq \max_{s\in S} f(s)\right]^{\frac{1}{2+4/\epsilon}} du\\
    &  ~ ~~ ~ ~ ~ ~ ~ ~ ~ ~\cdot\int_{\R}\vert H_{k-1}(u')\vert\Phi^+(u') P\left[ \min_{s\in S} f(s) < u \right]^{\frac{1}{2+4/\epsilon}} \,du' < \infty
\end{align*}

Case $\E\big[ (I_dIII_{d'} \big]$ and $\E\big[ (II_dIII_{d'} \big]$: basically, the same arguments as previous one.

Thus, by Fubini we have $\sigma_{dd'}<\infty$. \qed

\subsection{Proof of Corollary \ref{cor:CovDecay}}
 The computations in the proof of Theorem \ref{thm:LKC-unbiased} 2.) are valid for all $d,k\in\mathbb{N}$ and imply that the identity \eqref{eq:cov-L} is true, i.e. interchanging the integrals is justified. This, implies
 \begin{equation*}
    \frac{(2\pi)^{d/2} (2\pi)^{d'/2}}{(d-1)! (k-1)!} \iint H_{d-1}(u) \, H_{k-1}(v) \, \Cov\left[\chi_f(u), \chi_f(v)\right] \,du\,dv <\infty
 \end{equation*}
 for all $d,k\in \mathbb{N}$. However, this can only be true, if $\Cov\left[\chi_f(u), \chi_f(v)\right]$ decays faster for $u,u'\rightarrow \pm\infty$ than any polynomial in $u,u'$.\qed

\subsection{Proof of Theorem \ref{thm:bHPE}}
  \noindent (i) By the assumptions we can apply Theorem \ref{thm:UnbiasedAndVariance}(i) to the GMF $R^{(N)}_{\mathbf{g}}$, which yields that $\E\big[ R\vphantom{R}^{(N)}_{\mathbf{g}} \big] =\bm{{\L}}\big( \hat{\mathfrak{r}}\vphantom{r}^{(N)} \big) <\infty$ for almost all $R_1,...,R_N$. Thus, since $R^{(N)}_{\mathbf{g}_1},...,^{(N)}_{\mathbf{g}_N}$ are independent, the SLLN implies the claim.
  
  \noindent (ii.) Note that the LKCs as defined in \cite[eq. 12.4.7]{RFG:2007} and as they appear in the GKF are integrals of continuous functions over $S$. The volume elements $\mathcal{H}_j$ do depend only on the Riemannian metric induced by the random field, which is in local coordinates $\partial_1,...,\partial_D$ given by
    $$\hat{g}_{dd'}(s)=\E\left[ \partial_d R^{(N)}_{\mathbf{g}}(s) \partial_{d'}R^{(N)}_{\mathbf{g}}(s)\right] = \partial_{d}\partial_{d'}\hat{\mathfrak{r}}\vphantom{r}^{(N)}(s,s')\vert_{s=s'}.$$
  The latter does converge by \textbf{(R2)} almost surely uniformly to
    $$\partial_{d}\partial_{d'}{\mathfrak{r}}(s,s')\vert_{s=s'} =\E\left[ \partial_d G(s) \partial_{d'}G(s)\right] = {g}_{dd'}(s).$$
  The only other random quantity in the integral is the Riemannian curvature tensor, which by \cite[p.308, first equation]{RFG:2007} is a continuous function depending on
  \begin{align*}
   \E\left[ \partial_d\partial_{d'} R^{(N)}_{\mathbf{g}}(s) \partial_{d''}\partial_{d'''}R^{(N)}_{\mathbf{g}}(s)\right] &= \partial_{d}^{s}\partial_{d'}^{s}\partial_{d''}^{s'}\partial_{d'''}^{s'}\hat{\mathfrak{r}}\vphantom{r}^{(N)}(s,s')\vert_{s=s'}\\
   &\quad\quad\quad\quad\xrightarrow{~a.s.~} \partial_{d}^{s}\partial_{d'}^{s}\partial_{d''}^{s'}\partial_{d'''}^{s'}{\mathfrak{r}}(s,s')\vert_{s=s'}
  \end{align*}
  for all $d,d',d'', d'''\in\{1,...,D\}$, where the supscript in $\partial_{d}^{s}$ indicates to which of the two components the partial derivative is applied to.
  
  Since all the above convergences are uniformly almost surely, we can apply Lebesgue's dominated convergence theorem in order to interchange the limit $N$ tending to infinity and the integrals in order to obtain the claim
  $\bm{{\L}}\big( \hat{\mathfrak{r}}\vphantom{r}^{(N)} \big) \xrightarrow{~a.s.~} \bm{{\L}}\big( \mathfrak{r} \big)$.
\qed

\subsection{Proof of Theorem \ref{thm:ECcurveCLT}}
We only proof part (ii), which is a simple application of Lemma \ref{lemma:fintefCLT}. Part (i) is similarily a consequence of the continuous mapping theorem.

Let us endow $\mathbb{R}^D$ with its standard norm denoted by $\vert\cdot\vert$. We define the space of bounded functions $\ell^\infty(\mathbb{R})$ over $\mathbb{R}$ as the set of all functions satisfying
\begin{equation}\label{eq:normloo}
	\Vert f \Vert_\infty = \sup_{s\in\mathbb{R}} \vert f(s) \vert <\infty\,.
\end{equation}
Note that $\ell^\infty(\mathbb{R})$ is a metric space with metric $d(f,g) = \Vert f-g \Vert_\infty$, see \citet[Chapter 1.5]{Vaart:1996}.
\begin{lemma}\label{lemma:fintefCLT}
	 Let $\mathcal{F}\subset \ell^\infty(\mathbb{R})$ be the finite dimensional subspace spanned by the functions $\phi_1,...,\phi_D$ and assume that the norm on $\mathcal{F}$ is induced by \eqref{eq:normloo}. Let ${\bf{a}}^{(N)}\in\mathbb{R}^D$ be a sequence of mean zero random vectors satisfying $N^{-1/2}{\bf{a}}^{(N)}\xrightarrow{~D~}N(0,\Sigma)$ with finite covariance matrix $\Sigma\in \mathbb{R}^{D\times D}$. Then
	\begin{equation}
		N^{-1/2} \sum_{d=1}^D{a}^{(N)}_d \phi_d \xrightarrow{~D~} \sum_{d=1}^D X_d \phi_d
	\end{equation}
with ${\bf X}=(X_1,...,X_d)\sim N(0,\Sigma)$ weakly in $(\ell^\infty(\mathbb{R}), \Vert\cdot\Vert_\infty)$.
\end{lemma}
\begin{proof}
	Define the linear map $g:(\mathbb{R}^D, \Vert\cdot\Vert) \rightarrow (\mathcal{F}, \Vert\cdot\Vert_\infty)$ with $g(x)=\sum_{d=1}^D x_d\phi_d$. By linearity and the fact that $\mathbb{R}^D$ is finite, we have that $g$ is continuous. Hence the continuous mapping theorem \citep[p.20, 1.3.6 Theorem]{Vaart:1996} implies that
\begin{equation}
	N^{-1/2} \sum_{d=1}^D{a}^{(N)}_d \phi_d = N^{-1/2}g\big({\bf{a}}^{(N)}\big)= g\big(N^{-1/2}{\bf{a}}^{(N)}\big) \xrightarrow{~D~} g(X)  = \sum_{d=1}^D X_d \phi_d\,,
\end{equation}
in $(\mathcal{F}, \Vert\cdot\Vert_\infty)$, where ${\bf X}=(X_1,...,X_d)\sim N(0,\Sigma)$. This shows the claim, because weak convergence in $(\mathcal{F}, \Vert\cdot\Vert_\infty)$ is equivalent to weak convergence in its supspace $(\ell^\infty(\mathbb{R}), \Vert\cdot \Vert_\infty)$ for random variables with support in $(\mathcal{F}, \Vert\cdot\Vert_\infty)$, see \citet[p.24, 1.3.10 Theorem]{Vaart:1996}.
\end{proof}

\subsection{Proof of Theorem \ref{thm:CLT-threshold}}

\begin{proof}
	The result follows from applying Lemma \ref{lemma:inverse-delta-method} below with $\hat{g}(u) = \widehat{\EEC}\vphantom{EEC}^{(N)}\!(u) - \alpha$ and $g(u) = \EEC(u) - \alpha$.
The conditions of the lemma are satisfied by these functions by Theorem \ref{thm:ECcurveCLT}.	
\end{proof}

\begin{lemma}
	\label{lemma:inverse-delta-method}
	Let the function $\hat{g}^{(N)}(u)$ and its first derivative $\hat{g}'^{(N)}(u)$ be uniformly consistent estimators of the function $g(u)$ and its first derivative $g'(u)$, respectively, where both are uniformly continuous over $u\in\R$. Assume there exists an open interval $I=(a, b)$ such that $g$ is strictly monotone on $I$ and there exists a unique solution $u_0\in I$ to the equation $g(u)=0$. Define $\hat{u}^{(N)} = \sup \{u\in I: \hat{g}^{(N)}(u) = 0\}$.
	\begin{enumerate}
		\item[(i)] $\hat{u}$ is a consistent estimator of $u_0$.
		\item[(ii)] Suppose the derivative $g'(u_0) \ne 0$. If $\sqrt{N}[\hat{g}^{(N)}(u) - g(u)]$ converges uniformly in distribution to a Gaussian process $G(u)$, then $\sqrt{N}(\hat{u}^{(N)} - u_0)$ converges in distribution to $G(u_0)/g'(u_0)$.
	\end{enumerate}
\end{lemma}

\begin{proof}
\hfill
	\begin{enumerate}
		\item[(i)] Without loss of generality, assume that $g$ is strictly decreasing on $I$. For any $\varepsilon>0$, $g(u_0-\varepsilon)>0>g(u_0+\varepsilon)$ since $g(u_0)=0$. Since $\hat{g}^{(N)}(u)$ is a consistent estimator of $g(u)$, we have
		\begin{equation*}\label{g-hat-root}
		\P\left[ \hat{g}^{(N)}(u_0-\varepsilon)>0>\hat{g}^{(N)}(u_0+\varepsilon) \right] \to 1,
		\end{equation*}
		implying that with probability tending to 1, there is a root of $\hat{g}$ in $I_{0,\varepsilon} = (u_0-\varepsilon, u_0+\varepsilon)$. On the other hand, by the monotonicity of $g$, there exists $\delta>0$ such that 
		\[
		\inf_{u\in I\setminus I_{0,\varepsilon}} |g(u)| > \delta.
		\]
		Further, by the uniform consistency of $\hat{g}$,
		\[
		\P\left[ \sup_{u\in I} |\hat{g}^{(N)}(u) - g(u)|< \delta/2 \right] \to 1.
		\]
		Therefore, since
		\[
		\inf_{u\in I\setminus I_{0,\varepsilon}} |\hat{g}^{(N)}(u)|
		\ge \inf_{u\in I\setminus I_{0,\varepsilon}} |g(u)| -
		\sup_{u\in I\setminus I_{0,\varepsilon}} |\hat{g}^{(N)}(u) - g(u)|,
		\]
		we have that
		\begin{equation*}\label{g-hat-no-root}
		\begin{aligned}
		&\P\left[ \inf_{u\in I\setminus I_{0,\varepsilon}} |\hat{g}^{(N)}(u)| > \delta/2 \right] \\
		&\quad\quad\quad\quad\ge \P\left[ \inf_{u\in I\setminus I_{0,\varepsilon}} |g(u)| -
		\sup_{u\in I\setminus I_{0,\varepsilon}} |\hat{g}^{(N)}(u) - g(u)| > \delta/2 \right] \\
		&\quad\quad\quad\quad= \P\left[ \sup_{u\in I\setminus I_{0,\varepsilon}} |\hat{g}^{(N)}(u) - g(u)| < \inf_{u\in I\setminus I_{0,\varepsilon}} |g(u)| - \delta/2 \right]
		\to 1.
		\end{aligned}
		\end{equation*}
		This implies that with probability tending to 1, there is no root of $\hat{g}^{(N)}$ outside $I_{0,\varepsilon}$. From the definition of $\hat{u}^{(N)}$, we obtain that $\hat{u}^{(N)}$ is the only root of $\hat{g}^{(N)}$ in $I$ with probability tending to 1. Thus
		\[
		\P[|\hat{u}^{(N)}-u_0|<\varepsilon] = \P[\hat{u}^{(N)} \in I_{0,\varepsilon}] \to 1,
		\]
		yielding that $\hat{u}$ is a consistent estimator of $u_0$.
		
		\item[(ii)] By a Taylor expansion of $\hat{g}^{(N)}(u)$ around $u_0$,
		\[
		0 = \hat{g}^{(N)}(\hat{u}) = \hat{g}^{(N)}(u_0) + (\hat{u}^{(N)} - u_0) \hat{g}'^{(N)}(u^*),
		\]
		where $u^*$ is between $u_0$ and $\hat{u}$, i.e. $|u^* - u_0| \le |\hat{u}^{(N)} - u_0|$. Rearranging, and since $g(u_0) = 0$,
		\begin{equation}\label{eq:uhat-u0}
		\sqrt{N}(\hat{u}^{(N)} - u_0)
		= - \frac{\sqrt{N}[\hat{g}^{(N)}(u_0) - g(u_0)]}{\hat{g}'^{(N)}(u^*)}.
		\end{equation}
		The numerator converges to $G(u)$ in distribution by assumption. To see that the denominator converges to $g'(u_0)$ in probability, 
		\begin{equation}\label{eq:ghat-ustar}
		\begin{split}
		\hat{g}'^{(N)}(u^*) &= g'(u_0) + [\hat{g}'^{(N)}(u^*) - g'(u_0)]\\
		&=g'(u_0) + [\hat{g}'^{(N)}(u^*) -g'(u^*)] +[g'(u^*)- g'(u_0)] \to g'(u_0),
		\end{split}
		\end{equation}
		since $\hat{g}'^{(N)}(u^*) -g'(u^*)$ and $g'(u^*)- g'(u_0)$ converge to 0 in probability by the uniform consistency of $\hat{g}'^{(N)}$ and by part (i) of the lemma, respectively. The result follows immediately from \eqref{eq:uhat-u0} and \eqref{eq:ghat-ustar}.
	\end{enumerate}
\end{proof}

\end{document}